\documentclass[12pt]{amsart}

\usepackage{amsmath}
\usepackage{amssymb}
\usepackage{amsfonts}
\usepackage{amsthm}
\usepackage{enumerate}
\usepackage{color}
\usepackage{psfrag}
\usepackage{mathrsfs}
\usepackage{amscd}
\usepackage[hyperindex,backref=page]{hyperref}
\usepackage{graphicx,color}
\usepackage{tikz} 
\usepackage{enumitem}
\usepackage[utf8]{inputenc}

\newtheorem{thm}{Theorem}[section]
\newtheorem{cor}[thm]{Corollary}
\newtheorem{lemma}[thm]{Lemma}
\newtheorem{prop}[thm]{Proposition}

\newtheorem{q}[thm]{Question}

\theoremstyle{definition}

\newtheorem{ex}[thm]{Example}

\newtheorem{const}[thm]{Construction}

\theoremstyle{remark}
\newtheorem{rem}[thm]{Remark}

\begin{document}


\title[The Serre depth and the depth of Stanley--Reisner Rings]{The Serre depth of Stanley--Reisner Rings and the depth of their symbolic powers}

\author{Yuji Muta and Naoki Terai}

\address[Y. Muta]{Department of Mathematics, Okayama University, 3-1-1 Tsushima-naka, Kita-ku, Okayama 700-8530 Japan.}

\email{p8w80ole@s.okayama-u.ac.jp}

\address[N. Terai]{Department of Mathematics, Okayama University, 3-1-1 Tsushima-naka, Kita-ku, Okayama 700-8530, Japan.}
	
\email{terai@okayama-u.ac.jp}

\begin{abstract}
We investigate an invariant, called the Serre depth, from the perspective of combinatorial commutative algebra.  In this paper, we establish several properties of an analogue of the depth of Stanley--Reisner rings. In particular, we relate the Serre depth both to the minimal free resolution of a Stanley--Reisner ring and to that of its Alexander dual. Also, we establish an analogue of a known result that describes the depth of Stanley--Reisner rings in terms of skeletons. Moreover, we study the Serre depth for $(S_{2})$ and the depth on the symbolic powers of Stanley--Reisner ideals. It had been an open question whether the depth of the symbolic powers of Stanley--Reisner ideals satisfies a non-increasing property, but Nguyen and Trung provided a negative answer. We construct an example that the Serre depth for $(S_{2})$ and the depth do not satisfy this property and its second symbolic power is Cohen--Macaulay. Moreover, we prove that the sequence of the Serre depth for $(S_{2})$ on the symbolic powers is convergent and that its limit coincides with the minimum value. Finally, we study the Serre depth on edge and cover ideals. Whether the depth on symbolic powers of edge ideals satisfies a non-increasing property has remained an open question. We address a related problem and show that the Serre depth for $(S_{2})$ on edge ideals of any well-covered graph satisfies a non-increasing property. In addition, we prove that the Serre depth for $(S_{2})$ on the cover ideals of any graph also satisfies a non-increasing property. Moreover, we determine the Serre depth on edge ideals of very well-covered graphs. 
\end{abstract}






\thanks{}
	

\maketitle


\section{Introduction}
In commutative ring theory, Serre's condition $(S_{r})$ provides a natural and essential generalization of the Cohen–Macaulay property, which has been a foundational concept for many years. Let us recall the definition of Serre's condition $(S_{r})$. For a Noetherian ring $R$, an $R$-module $M$ and a positive integer $r$, $M$ satisfies {\it Serre's condition} $(S_{r})$ if the inequality ${\rm depth}(M_{\mathfrak{p}})\geq\min\{r,\dim M_{\mathfrak{p}}\}$ holds for every $\mathfrak{p}\in{\rm Spec}(R)$. Many researchers have been studying Serre's condition $(S_{r})$ in combinatorial commutative algebra. (\cite{dlv, gpfy, hk, h2,  mt2, pfty, ppty1,ppty2, rty, sf2, ty}). 

In this paper, we focus on an invariant called the Serre depth, which was introduced in \cite{ppty1, ppty2}. Let us recall the definition of the Serre depth. Let $S=\Bbbk[x_{1},\ldots, x_{n}]$ be a polynomial ring in $n$ variables over an arbitrary field $\Bbbk$, $\mathfrak{m}=(x_{1},\ldots, x_{n})$ be the unique homogeneous maximal ideal and let $M$ be an unmixed finitely generated graded $S$-module, $r\geq 2$. Then, the {\it Serre depth} for $(S_{r})$, denoted by $S_{r}\mbox{-}{\rm depth}(M)$, is defined as 
$$S_{r}\mbox{-}{\rm depth}(M)=\min\{j\,\,: \dim K_{M}^{j}\geq j-r+1\},$$
where $K_{M}^{j}={\rm Hom}_{\Bbbk}(H_{\mathfrak{m}}^{j}(M),\Bbbk)$ and $\dim K_{M}^{j}$ denotes the Krull dimension of $K_{M}^{j}$ as an $S$-module. Here, if $K_{M}^{j}=0$, then we set $\dim K_{M}^{j}=-\infty$. By the local duality theorem (see e.g. \cite[Theorem 3.6.19]{bh2}), we have an isomorphism$$K_{M}^{j}\simeq{\rm Ext}_{S}^{n-j}(M, \omega_{S}),$$where $\omega_{S}$ denotes the canonical module of $S$. By \cite[Lemma 3.2.1]{s0}, one can see that $S_{r}\mbox{-}{\rm depth}M=\dim M$ if and only if $M$ satisfies $(S_{r})$. Also, by the definition of the Serre depth, one can see that$$\dim M\geq S_{2}\mbox{-}{\rm depth}(M)\geq\cdots\geq S_{d}\mbox{-}{\rm depth}(M)={\rm depth}(M),$$where $d=\dim M$. Therefore, the Serre depth is defined as an analogue of the classical notion called the depth: just as the depth measures how far a module is from being Cohen--Macaulay, the Serre depth does so with respect to Serre's condition $(S_{r})$. There has been research on the Serre depth in \cite{kmt, ppty1, ppty2}. 

After introduction and preliminaries, in Section \ref{sqr}, we study the Serre depth of Stanley--Reisner rings and its relation to their minimal free resolutions. First, we relate the shape of a minimal free resolution of a Stanley--Reisner ring to its Serre depth: 
\begin{thm}[\mbox{\rm see, Theorem \ref{length-mfr}}]
Let $\Delta$ be a pure simplicial complex on $[n]$. Then the maximal length of the $j$-th linear part of the minimal graded free resolution of $\Bbbk[\Delta]$ for $j<r$ is less than or equal to $n-S_{r}\mbox{-}{\rm depth}(\Bbbk[\Delta]),$ that is, 
$$\max\{i\,\,: \beta_{i,i+j}(\Bbbk[\Delta])\neq 0\mbox{ for all }j<r\}\leq n-S_{r}\mbox{-}{\rm depth}(\Bbbk[\Delta]).$$
\end{thm}

Moreover, we provide the following characterization of the Serre depth of the Stanley--Reisner ring of a simplicial complex in terms of the minimal free resolution of the Stanley--Reisner ring of its Alexander dual complex as follows:

\begin{thm}[\mbox{\rm see, Theorem \ref{eq dual}}]\label{eq dual intro}
Let $\Delta$ be a pure simplicial complex on $[n]$ and let $r\geq2$. Then we have 
$$\max\{j\mid \beta_{i,i+j}(\Bbbk[\Delta^{\vee}])\neq 0\mbox{ for some }i\leq r\}= n-S_{r}\mbox{-}{\rm depth}(\Bbbk[\Delta])-1.$$
\end{thm}

By using this characterization of the Serre depth, we generalize the results on the depth of Stanley--Reisner rings due to Terai \cite[Corollary 0.3]{t} and Yanagawa \cite[Corollary 3.7]{y} with Terai to the case that of the Serre depth. These results are themselves generalizations of the classical result by Eagon and Reiner \cite[Theorem 3]{er}. Motivated by this, we establish the following corollary, which unifies and extends the above results:
\begin{cor}[\mbox{\rm see, Corollary \ref{gene}}]
For a pure simplicial complex $\Delta$, we have  
$${\rm reg}_{\leq r-1}I_{\Delta^{\vee}}-{\rm indeg}I_{\Delta^{\vee}}=\dim \Bbbk[\Delta]-S_{r}\mbox{-}{\rm depth}(\Bbbk[\Delta]),$$where ${\rm reg}_{\leq r}I_{\Delta^{\vee}}=\max\{j\,\,:\beta_{i, i+j}(I_{\Delta^{\vee}})\neq 0\mbox{ for some }i\leq r\}.$
\end{cor}

Also, a beautiful result on the depth of Stanley--Reisner rings was proved by Smith \cite[Theorem 3.7]{s1}. We extend this result to the Serre depth by using Theorem \ref{eq dual intro}:
\begin{thm}[\mbox{\rm see, Theorem \ref{formula}}]
Let $\Delta$ be a pure simplicial complex with $\dim \Delta=d-1$ and $2\leq r\leq d$. Then, we have 
$$S_{r}\mbox{-}{\rm depth}(\Bbbk[\Delta])=1+\max\{i\mid \Bbbk[\Delta^{i}]\mbox{ satisfies Serre's condition }(S_{r})\},$$
where $\Delta^{i}=\{F\mid {\rm dim}F\leq i\}$ for $-1\leq i\leq \dim \Delta$. 
\end{thm}

Moreover, we compute the Serre depth of a Stanley--Reisner ring after applying an operation known as a 1-vertex inflation, based on the Serre depth prior to this operation (see, Theorem \ref{1-vertex}).

In Section \ref{mono}, we study the Serre depth on monomial ideals, including symbolic powers of Stanley--Reisner ideals. To investigate arbitrary monomial ideals via squarefree monomial ideals, the technique known as the polarization. We prove that, as with depth, the Serre depth is preserved under the polarization (see, Proposition \ref{polarization}).
Rinaldo, Terai, and Yoshida described when the quotient ring modulo the second symbolic power of a Stanley--Reisner ideal satisfies Serre's condition $(S_{2})$ in terms of simplicial complexes \cite[Corollary 3.3]{rty}. We generalize this result by using the Serre depth for $(S_{2})$ as follows: 
\begin{thm}[\mbox{\rm see, Theorem \ref{S2depth}}]
Let $\Delta$ be a pure simplicial complex with $ \dim \Delta=d-1$ and $2 \le s\le d$.
Then the following conditions are equivalent$:$
\begin{enumerate}
\item $\mbox{\rm $S_2$-depth}(S/I_{\Delta}^{(2)})\ge s$.
\item ${\rm diam} (({\rm link}_{\Delta} F)^{(1)}) \le 2$ for any face $F \in \Delta$ with 
$\dim {\rm link}_{\Delta} F \ge d-s+1$. 
\end{enumerate}
\end{thm}

In the rest of this section, we discuss the Serre depth for $(S_{2})$ and the depth on the symbolic powers of Stanley--Reisner ideals. There has been extensive research on the depth on the symbolic powers (\cite{cpfty, hlt, hktt, kty, mmvv, nt, rty, sf1, sf3, tt, v0}). In Theorem \ref{1st local coho}, we give a necessary and sufficient condition for the vanishing of the first local cohomology module on the symbolic powers of Stanley–Reisner ideals in terms of simplicial complexes. Accordingly, we obtain that the combinatorial characterization of the Serre depth for $(S_{2})$ on the symbolic powers of Stanley--Reisner ideals (see, Corollary \ref{combi of S_2-depth}). Moreover, we generalize

Also, it had been an open question whether the depth on the symbolic powers of Stanley–Reisner ideals satisfies a non-increasing property, but Nguyen and Trung provided a negative answer (\cite[Theorem 2.8]{nt}). In the following theorem, we construct a simplicial complex which is Cohen–Macaulay and whose the second symbolic power is also Cohen–Macaulay, but the depth does not satisfy this property. We also show that the Serre-depth for $(S_{2})$ does not satisfy this property. Note that  squarefree monomial ideals constructed by Nguyen and Trung is not unmixed.

\begin{thm}[\mbox{\rm see, Theorem \ref{CE of non-increasing}}]
There exists a non-cone simplicial complex $\Delta$ with $\dim\Delta=d-1\geq 2$ such that the following conditions are satisfied:  
\begin{enumerate}
\item $\Delta$ is pure and shellable,  
\item $S/I_{\Delta}^{(2)}$ is Cohen--Macaulay, 
\item ${\rm depth}(S/I_{\Delta}^{(d+1)})=1\mbox{ and }{\rm depth}(S/I_{\Delta}^{(d+2)})\geq2$, 
\item $S_{2}\mbox{-}{\rm depth}(S/I_{\Delta}^{(d+1)})=1\mbox{ and }S_{2}\mbox{-}{\rm depth}(S/I_{\Delta}^{(d+2)})\geq2$. 
\end{enumerate}
\end{thm}

Moreover, we discuss the Serre depth for $(S_{2})$ and the depth on the symbolic powers of  Stanley--Reisner ideals when the dimension of a simplicial complex is 1. We prove that the Serre depth for $(S_{2})$ and the depth satisfies a non-increasing property in this case: 
\begin{thm}[\mbox{\rm see, Theorem \ref{non-increasing for dim1}}]
Let $\Delta$ be a simplicial complex with $\dim \Delta=1$ and $\ell\geq1$. Then we have $${\rm depth}(S/I_{\Delta}^{(\ell)})\geq{\rm depth}(S/I_{\Delta}^{(\ell+1)}).$$Moreover, if $\Delta$ is pure, then we have$$S_{2}\mbox{-}{\rm depth}(S/I_{\Delta}^{(\ell)})\geq S_{2}\mbox{-}{\rm depth}(S/I_{\Delta}^{(\ell+1)}).$$
\end{thm}
As a corollary, we completely classify the Serre depth for $(S_{2})$ and the depth of the symbolic powers of Stanley–Reisner ideals when the dimension of a simplicial complex is 1:  
\begin{cor}[\mbox{\rm see, Corollary \ref{classify of the depth}}]
Let $\Delta$ be a simplicial complex with $\dim \Delta=1$. Then the sequence of the depth on symbolic powers $({\rm depth}(S/I_{\Delta}^{(\ell)}))_{\ell\geq1}$ is as follows: 
$$({\rm depth}(S/I_{\Delta}^{(\ell)}))_{\ell\geq1}=
\begin{cases}
(2,2,2,\ldots) & \mbox{ if }\Delta\mbox{ is matroid}, \\
(2,2,1,\ldots) & \mbox{ if }{\rm diam}\Delta\leq2\mbox{ and }\Delta\mbox{ is not matroid}, \\
(2,1,1,\ldots) & \mbox{ if }3\leq{\rm diam}\Delta<\infty, \\
(1,1,1,\ldots) & \mbox{ if }{\rm diam}\Delta=\infty. 
\end{cases}
$$Moreover, if $\Delta$ is pure, then we have $(S_{2}\mbox{-}{\rm depth}(S/I_{\Delta}^{(\ell)}))_{\ell\geq1}=({\rm depth}(S/I_{\Delta}^{(\ell)}))_{\ell\geq1}$. 
\end{cor}

Also, we obtain that if the number of vertices of a simplicial complex is less than or equal to 5, then the Serre depth for $(S_{2})$ and the depth on the symbolic powers of its Stanley–Reisner ideal satisfy a non-increasing property, otherwise, there exists an example that a non-increasing property does not hold (see, Proposition \ref{cor of dim1 case} and Example \ref{CE2}). Moreover, we prove that the sequence of the Serre depth for $(S_{2})$ on the symbolic powers is convergent and that its limit coincides with the minimum value: 
\begin{thm}[\mbox{\rm see, Theorem \ref{convergent}}]
For a pure simplicial complex $\Delta$, the sequence $\{S_{2}\mbox{-}{\rm depth}(S/I_{\Delta}^{(\ell)})\}_{\ell\geq1}^{\infty}$ is convergent. Moreover, we have $$\min_{\ell}\{S_{2}\mbox{-}{\rm depth}(S/I_{\Delta}^{(\ell)})\}=\lim_{\ell\rightarrow\infty}S_{2}\mbox{-}{\rm depth}(S/I_{\Delta}^{(\ell)}).$$
\end{thm}

In the final section of this paper, we treat the Serre depth on the edge and cover ideals of graphs. As a corollary of Theorem \ref{eq dual intro}, we reformulate a result in \cite[Corollary 2.2]{th} and determine the Serre depth on the cover ideals of a special class of graphs. It has been an open question whether the depth on the symbolic powers of edge ideals satisfies the non-increasing property (cf. \cite[Problem 1.2]{kty}, \cite[Question 2.9]{nt}). As a natural analogue of this important problem, while it remains open for the depth, we prove that a non-increasing property of the Serre depth for $(S_{2})$ holds for the edge ideals of any well-covered graph. Note that well-coveredness means the edge ideal is unmixed. 

\begin{thm}[\mbox{\rm see, Theorem \ref{non-increasing of edge ideals}}]
For a well-covered graph $G$, we have$$S_{2}\mbox{-}{\rm depth}(S/I(G)^{(\ell)})\geq S_{2}\mbox{-}{\rm depth}(S/I(G)^{(\ell+1)})\mbox{ for all }\ell.$$
\end{thm}

Also, we prove that a non-increasing property of the Serre depth for $(S_{2})$ holds for cover ideals of any simple graph (see, Proposition \ref{non-increasing of cover ideals}), as an analogue of \cite[Theorem 3.2]{hktt}. Moreover, for any $r \geq 2$, we prove that a non-increasing property of the Serre depth for $(S_{r})$ holds for cover ideals of graphs belong to a special class (see, Theorem \ref{non-increasing for S_r}). Finally, we determine the Serre depth for $(S_{r})$ on the edge ideal of a very well-covered graphs by using the structure theorem of very well-covered graphs proved in \cite[Theorem 3.5]{kpty} (see, Theorem \ref{S_r very well}). 


\section{Preliminaries}
In this section, we present basic terminology and notation used in this paper. We refer the reader to \cite{bh2, hh, s4, v} for more detailed information. Readers who are already familiar with these concepts may skip to the next section. 

\subsection{Stanley--Reisner rings}
Let $n$ be a positive integer and let $[n] = \{1, \ldots, n\}$. A {\it simplicial complex} $\Delta \subset 2^{[n]}$ is a nonempty collection of subsets of $[n]$ that is closed under taking subsets, that is, if $F \in \Delta$ and $G \subset F$, then $G \in \Delta$. An element of $\Delta$ is called a {\it face} of $\Delta$. A face is called a {\it facet} if it is maximal with respect to inclusion. If the facets of $\Delta$ are $F_1, \ldots, F_r$, we write $\Delta = \langle F_1, \ldots, F_r \rangle$. For a subset $W$ of $[n]$, let $\Delta_{W}$ be the simplicial complex consisting all faces of $\Delta$ that are contained in $W$. Also, $\Delta$ is called a \textit{matroid complex}, if the induced subcomplex $\Delta_{W}$ is pure for any $W \subset [n]$. We do not assume that the singleton $\{i\}\in\Delta$ for all $i\in[n]$. Also, for a face $F$, the {\it link} of $\Delta$ with respect to $F$, denoted by ${\rm link}_{\Delta}F$ is defined as 
$${\rm link}_{\Delta}F=\{G\in\Delta\,\,: G\cup F\in\Delta, G\cap F=\emptyset\}.$$The {\it Stanley--Reisner ideal} $I_{\Delta}$ of $\Delta$ is the ideal of $S$ generated by monomials correspond to minimal non-faces of $\Delta$, that is, 
$$I_{\Delta}=(x_{i_{1}}\cdots x_{i_{r}}\,\,: \{i_{1},\ldots, i_{r}\}\mbox{ is a minimal non-face of }\Delta).$$Also, $\Bbbk[\Delta]=S/I_{\Delta}$ is called the {\it Stanley--Reisner ring} of $\Delta$. The {\it Alexander dual complex }$\Delta^{\vee}$ of $\Delta$ defined as $$\Delta^{\vee}=\{F\in 2^{[n]}\,\,: [n]\setminus F\notin\Delta\}.$$Suppose that $\Delta$ is connected with $\dim \Delta=1$ and let $p$, $q$ be two vertices. The \textit{distance} between $p$ and $q$, denoted by ${\rm dist}_{\Delta}(p, q)$, is the minimal length of paths from $p$ to $q$. The \textit{diameter}, denoted by ${\rm diam}\Delta$, is the maximal distance between two vertices in  $[n]$. We set ${\rm diam}\Delta= \infty$, if $\Delta$ is  disconnected. Also, let us recall the notion of symbolic powers. Given an integer $i\geq 1$ and a homogeneous ideal $I$, the \textit{$i$-th symbolic power of $I$} is defined to be an ideal $$I^{(i)}=S\cap\bigcap_{P\in{\rm Ass}(I)}I^{i}S_{P}.$$If $I$ is a squarefree monomial ideal, then one has $$I^{(i)}=\bigcap_{P\in{\rm Ass}(I)}P^i.$$

\subsection{Edge ideals and Cover ideals}
Throughout, we assume that graphs mean finite simple graphs, that is, finite undirected graphs without loops or multiple edges and let $\Bbbk$ be an arbitrary field. Let $G=(V(G), E(G))$ be a graph, where $V(G)$ denotes the vertex set of $G$, $E(G)$ denotes the edge set of $G$, and $\Bbbk[V(G)]=\Bbbk[x_1,\ldots, x_{n}]$ be a polynomial ring over $\Bbbk$. Then the {\it edge ideal}, denoted by $I(G)$, is the ideal  of $\Bbbk[V(G)]$ defined by $I(G)=(x_{i}x_{j}\,\,: \{x_{i}, x_{j}\}\in E(G)).$ Also, the {\it cover ideal} of $G$, denoted by $J(G)$, is the ideal defined by $J(G)=(x_{i_{1}}\cdots x_{i_{s}}\,\,:\{i_{1},\ldots, i_{s}\}\mbox{ is a vertex cover of }G).$ Cover ideals correspond to the Alexander dual ideals of edge ideals. A subset $C$ of $V(G)$ is called {\it vertex cover}, if $C\cap e\neq\emptyset$ for every edge $e$ of $G$. In particular, if $C$ is a vertex cover that is minimal with respect to inclusion, then $C$ is called a {\it minimal vertex cover} of $G$. If every minimal vertex cover of $G$ has the same cardinality, then $G$ is called {\it well-covered}. A subset $A$ of $V(G)$ is called an {\it independent set}, if no two vertices in $A$ are adjacent to each other. In particular, if $A$ is an independent set of $G$ that is maximal with respect to inclusion, then $A$ is called a {\it maximal independent set} of $G$. The {\it independence complex} of $G$ is the set of independent sets of $G$, which forms a simplicial complex $\Delta(G)$. It is known that the edge ideal of $G$ coincides with the Stanley--Reisner ideal of $\Delta(G)$ (see, for example, \cite[p. 73, Lemma 31]{v0}). A subset $M$ of $E(G)$ is called a {\it matching}, if no two  edges in $M$ share a common vertex. Moreover, if there is no edge in $E(G)\setminus M$ that is contained in the union of edges of $M$, then $M$ is called an {\it induced matching}. Then, the {\it induced matching number} of $G$, denoted by ${\rm im}(G)$, is defined as $${\rm im}(G)=\max\{|M|\,\,: M\mbox{ is an induced matching of }G\}.$$

\subsection{Minimal free resolutions and local cohomology modules}
Let $\Bbbk$ be any field, and set $S = \Bbbk[x_1, \ldots, x_n]$ to be the polynomial ring in $n$ variables over $\Bbbk$. For a finitely generated graded $S$-module $M$, there exists a {\it graded minimal free resolution} of $M$ of the form$$0 \rightarrow \bigoplus_{j \in \mathbb{Z}} S(-j)^{\beta_{p,j}(M)} \rightarrow \cdots \rightarrow \bigoplus_{j \in \mathbb{Z}} S(-j)^{\beta_{0,j}(M)} \rightarrow M \rightarrow 0,$$where $S(-j)$ denotes the graded $S$-module defined by $S(-j)_k = S_{k - j}$ for all $k \in \mathbb{Z}$. The number $\beta_{i,j}(M)$, referred to as the $(i,j)$-th {\it graded Betti number} of $M$, is an invariant that represents the number of minimal homogeneous generators of degree $j$ appearing in the $i$-th syzygy module in a minimal graded free resolution of $M$. Then the {\it Castelnuovo--Mumford regularity}, denoted by ${\rm reg}(M)$, is defined as 
$${\rm reg}(M)=\max\{j-i\,\,:\beta_{i,j}(M)\neq0\}.$$
Also the {\it projective dimension} of $M$, denoted by ${\rm pd}(M)$ is defined as 
$${\rm pd}(M)=\max\{i\,\,:\beta_{i,j}(M)\neq0\mbox{ for some }j\}.$$For Stanley--Reisner ideals, by Hochster \cite{h1}, it is known that the combinatorial description of graded Betti numbers and the Hilbert series of the local cohomology modules of Stanley--Reisner rings Moreover, by using a technique known as ``local Alexander duality'' shown by Eagon and Reiner \cite[proof of Proposition 1]{er}, we have 
\begin{thm}[\cite{h1, er}]
For a simplicial complex $\Delta$ on the vertex set $[n]$, the $(i,j)$-th graded Betti number of $\Bbbk[\Delta]$ is given by 
$$\beta_{i,j}(\Bbbk[\Delta])=\displaystyle \sum_{\substack{W\subset[n], |W|=j}}{\rm dim}_{\Bbbk}\widetilde{H}_{j-i-1}(\Delta_{W};\Bbbk)$$Moreover, by using ``local Alexander duality'', we have $$\beta_{i,j}(\Bbbk[\Delta])=\displaystyle \sum_{\substack{F\in\Delta^{\vee}, |F|=n-j}}{\rm dim}_{\Bbbk}\widetilde{H}_{i-2}({\rm link}_{\Delta^{\vee}}F;\Bbbk)$$
Also, the Hilbert series of the local cohomology modules of Stanley--Reisner rings given by 
\begin{align*}
F(H^i_{\mathfrak{m}}(\Bbbk[\Delta]),t)
& =\displaystyle\sum_j\dim_\Bbbk[H^i_{\mathfrak{m}}(\Bbbk[\Delta] )]_j \ t^j \\[0.15cm]
& =\displaystyle \sum_{F \in \Delta} \dim_\Bbbk\widetilde{H}_{i-|F|-1} ({\rm link}_\Delta F; \Bbbk)\left(\frac{t^{-1}}{1-t^{-1}}\right)^{|F|},
\end{align*}
Moreover, we may write 
$$F(K_ {S/I_{\Delta}}^ {i}, t) = \displaystyle\sum_{F\in\Delta} \widetilde {H}_ {i-|F|-1} ({\rm link}_ {\Delta} F; \mathbb {K}) \left(\frac{t} {1-t} \right)^ {|F|}.$$
\end{thm}


\section{The Serre depth of Stanley--Reisner rings}\label{sqr}
In this section, we establish several properties of the Serre depth of Stanley--Reisner rings.  
In particular, we associate the Serre depth with a minimal free resolution of the Stanley--Reisner ring of a simplicial complex and its Alexander dual. Firstly, we characterize the Serre depth of Stanley--Reisner rings with reduced homology groups. 
\begin{lemma}\label{depth}
Let $\Delta$ be a pure simplicial complex and let $r\geq 2$. Then the following statements are equivalent: 
\begin{enumerate}
\item $s\leq S_{r}$-{\rm depth}$(\Bbbk[\Delta])$, 
\item $\widetilde{H}_{i}({\rm link}_{\Delta}F;\Bbbk)=0$ for all $i\leq r-2$ and all $F\in\Delta$ with $|F|\leq s-i-2$.
\end{enumerate} 
\end{lemma}
\begin{proof}
By the definition of the Serre depth, (1) is equivalent to ${\rm dim}K_{\Bbbk[\Delta]}^{i}\leq i-r$ for all $i\leq s-1$. Also, by Hochster's formula for local cohomology modules, 
the latter one is equivalent to $\widetilde{H}_{i-|F|-1}({\rm link}_{\Delta}F;\Bbbk)=0$
for all $i\leq s-1$ and a face $F$ of $\Delta$ with $|F|\geq i-r+1$. Therefore, it is also equivalent to (2), as required. 
\end{proof}

We present a relationship between the Serre depth $S_{r}\mbox{-}{\rm depth}(\Bbbk[\Delta])$ and the length of a minimal free resolution of $\Bbbk[\Delta]$ as a generalization of the result in \cite{mt2}.

\begin{thm}\label{length-mfr}
Let $\Delta$ be a pure simplicial complex on $[n]$. Then the maximal length of the $j$-th linear part of the minimal graded free resolution of $\Bbbk[\Delta]$ for $j<r$ is less than or equal to $n-S_{r}\mbox{-}{\rm depth}(\Bbbk[\Delta]),$ that is, 
$$\max\{i\,\,: \beta_{i,i+j}(\Bbbk[\Delta])\neq 0\mbox{ for all }j<r\}\leq n-S_{r}\mbox{-}{\rm depth}(\Bbbk[\Delta]).$$
\end{thm}
\begin{proof}
Fix an integer $j\leq r$. We prove that $\beta_{i,i+j}(\Bbbk[\Delta])=0$ for all $i\geq n-S_{r}\mbox{-}{\rm depth}(\Bbbk[\Delta])$. 
Set $s=S_{r}\mbox{-}{\rm depth}(\Bbbk[\Delta])$ and fix an integer $i$ with $i\geq n-s$. 
From Hochster's formula on graded Betti numbers, we have 
$$\beta_{i,i+j}(\Bbbk[\Delta])=\displaystyle \sum_{\substack{W\subset[n], \\[0.1cm] |W|=i+j}}{\rm dim}_{\Bbbk}\widetilde{H}_{|W|-i-1}(\Delta_{W};\Bbbk).$$
Fix $W\subset[n]$ with $|W|=i+j$. It is enough to prove that ${\rm dim}_{\Bbbk}\widetilde{H}_{|W|-i-1}(\Delta_{{W}};\Bbbk)=0$. 
Now, by the proof in {\cite[Corollary 2.6]{bcp}}, we have 
$${\rm dim}_{\Bbbk}\widetilde{H}_{|W|-i-1}(\Delta_{W};\Bbbk)\leq\displaystyle  \sum_{\substack{F\in\Delta, \\[0.1cm] F\cap W=\emptyset}}{\rm dim}_{\Bbbk}\widetilde{H}_{|W|-i-1}({\rm link}_{\Delta}F;\Bbbk).$$
Hence, it suffices to show that ${\rm dim}_{\Bbbk}\widetilde{H}_{|W|-i-1}({\rm link}_{\Delta}F;\Bbbk)=0$ for all face $F$ of $\Delta$ with $F\cap W=\emptyset$. Fix a face $F$ of $\Delta$ with $F\cap W=\emptyset$. 
Now, since $|W|=i+j$, $|W|-i-1=j-1\leq r-2$. 
Also, since $F\cap W=\emptyset$, 
$|F|\leq n-|W|\leq s-(|W|-i-1)-2$. 
Thus, by Lemma \ref{depth}, 
 ${\rm dim}_{\Bbbk}\widetilde{H}_{|W|-i-1}({\rm link}_{\Delta}F;\Bbbk)=0$. 
Therefore, $\beta_{i, i+j}(\Bbbk[\Delta])=0$ for all $i\geq n-s$, as desired. 
\end{proof}

Thanks to Theorem \ref{length-mfr}, we have following corollaries: 
\begin{cor}
For a pure simplicial complex $\Delta$, if $\beta_{{\rm pd}(\Bbbk[\Delta]), {\rm pd}(\Bbbk[\Delta])+r}(\Bbbk[\Delta])\neq 0,$ then we have $S_{t}\mbox{-}{\rm depth}(\Bbbk[\Delta])={\rm depth}(\Bbbk[\Delta])$ for all $t>r$.
\end{cor}
\begin{proof}
From Theorem \ref{length-mfr} and the Auslander--Buchsbaum formula, we have  
\[
\begin{array}{lll}
n-S_{t}\mbox{-}{\rm depth}(\Bbbk[\Delta])
& \geq & \max\{i \mid \beta_{i,i+j}(\Bbbk[\Delta])\neq 0\mbox{ for some }j< t\} \\[0.2cm]
& = & {\rm pd}(\Bbbk[\Delta]) \\[0.2cm]
& = & n-{\rm depth}(\Bbbk[\Delta]) ,
\end{array}
\]
which completes the proof. 
\end{proof}
\begin{cor}\label{lin}
Let $\Delta$ be a pure simplicial complex and $r\geq 2$. If $\Bbbk[\Delta]$ has an $r$-linear resolution, then $S_{t}\mbox{-}{\rm depth}(\Bbbk[\Delta])={\rm depth}(\Bbbk[\Delta])$ for all $t>r$. 
\end{cor}

\begin{cor}
Let $\Delta$ be a pure $(d-1)$-dimensional simplicial complex on the vertex set $[n]$ and let $r\geq 2$ be an integer. 
If $\Delta$ satisfies Serre's condition $(S_{r})$, then the maximal length of the $i$-linear part of the minimal graded free resolution of $\Bbbk[\Delta]$ for $i< r$ is less than or equal to $n-d$. 
\end{cor}
\begin{proof}
Notice that $\Delta$ satisfies Serre's condition $(S_{r})$ if and only if $S_{r}$-{\rm depth}$(\Bbbk[\Delta])=d$. Hence, by Theorem \ref{length-mfr}, the assertion follows. 
\end{proof}

We relate the Serre depth of $\Bbbk[\Delta]$ to a minimal free resolution of $\Bbbk[\Delta^{\vee}]$. Using the following theorem, we derive several corollaries concerning the Serre depth of Stanley--Reisner rings, thereby generalizing known results about the depth.

\begin{thm}\label{eq dual}
Let $\Delta$ be a pure simplicial complex and let $r\geq2$. Then we have 
$$\max\{j\mid \beta_{i,i+j}(\Bbbk[\Delta^{\vee}])\neq 0\}\leq n-S_{r}\mbox{-}{\rm depth}(\Bbbk[\Delta])-1\mbox{ for all }i\leq r.$$
In particular, 
$$\max\{j\mid \beta_{i,i+j}(\Bbbk[\Delta^{\vee}])\neq 0\mbox{ for some }i\leq r\}= n-S_{r}\mbox{-}{\rm depth}(\Bbbk[\Delta])-1.$$
\end{thm}
\begin{proof}
Fix an integer $i\leq r$ and set $s=S_{r}\mbox{-}{\rm depth}(\Bbbk[\Delta])$. From Lemma \ref{depth}, we have $\widetilde{H}_{i-2}({\rm link}_{\Delta}F;\Bbbk)=0$ for all face $F$ of $\Delta$ with $|F|\leq s-i$. Notice that $n-j\leq s-i$, if $j\geq n-s+i$. Now, from Hochster's formula on graded Betti numbers, we have 
$$\beta_{i,j}(\Bbbk[\Delta^{\vee}])=\displaystyle \sum_{\substack{G\in\Delta, \\[0.1cm] |G|=n-j}}{\rm dim}_{\Bbbk}\widetilde{H}_{i-2}({\rm link}_{\Delta}G;\Bbbk)=0\mbox{ for all } j\geq n-s+i$$
Hence, $\beta_{i,i+n-s}(\Bbbk[\Delta^{\vee}])=\beta_{i,i+n-s+1}(\Bbbk[\Delta^{\vee}])=\cdots=0$. Also, since ${\rm dim}K_{\Bbbk[\Delta]}^{s}\geq s-r+1$, from Hochster's formula for local cohomology modules, there exists a face $G$ of $\Delta$ with $|G|\geq s-r+1$ such that $\widetilde{H}_{s-|G|-1}({\rm link}_{\Delta}G;\Bbbk)\neq0$. Then, by setting $|G|=s-r+l$ for $l\geq 1$, we have 
$$\beta_{r-l+1,(r-l+1)+n-s-1}(\Bbbk[\Delta^{\vee}])=\displaystyle\sum_{\substack{F\in\Delta, \\[0.1cm] |F|=s-r+l}}{\rm dim}_{\Bbbk}\widetilde{H}_{r-l-1}({\rm link}_{\Delta}F;\Bbbk)\neq 0,$$as required. 
\end{proof}

Before stating corollaries of Theorem \ref{eq dual}, let us recall the definition of the condition $(N_{c,r})$, which was introduced in \cite[Definition 3.6]{y}. For a homogeneous ideal $I$ of $S$, $I$ satisfies the condition $(N_{c,r})$ if $\beta_{i,j}(I)=0$ for all $i<r$ and $j\neq i+c$. 
For a simplicial complex $\Delta$ and $r\geq 2$,  
we set 
$${\rm reg}_{\leq r}I_{\Delta}=\max\{j\,\,:\beta_{i, i+j}(I_{\Delta})\neq 0\mbox{ for some }i\leq r\}. $$ 
Then, we have the following statement. 
\begin{prop}\label{N_c,r}
Let $\Delta$ be a simplicial complex. Then the followings are equivalent: 
\begin{enumerate}
\item ${\rm reg}_{\leq r-1}I_{\Delta}={\rm indeg}I_{\Delta}$
\item $I_{\Delta}$ satisfies $(N_{{\rm indeg}I_{\Delta},r})$
\end{enumerate}
\end{prop}
\begin{proof}
If $I_{\Delta}$ does not satisfy the condition $(N_{{\rm indeg}\,I_{\Delta},r})$, then there exist $i < r$ and $j \neq {\rm indeg}\,I_{\Delta}$ such that $\beta_{i,i+j}(I_{\Delta}) \neq 0$. Since $\beta_{i,i+j}(I_{\Delta}) = 0$ whenever $j < {\rm indeg}\,I_{\Delta}$, we may assume that $j > {\rm indeg}\,I_{\Delta}$. However, this leads to a contradiction, because$$\max\left\{ j \,:\, \beta_{i,i+j}(I_{\Delta}) \neq 0 \text{ for some } i \leq r - 1 \right\} = {\rm reg}_{\leq r-1}(I_{\Delta}) = {\rm indeg}\,I_{\Delta}.$$Conversely, if $I_{\Delta}$ satisfies the condition $(N_{{\rm indeg}\,I_{\Delta},r})$, then by the definition, we have $\beta_{i,i+j}(I_{\Delta}) = 0$ for all $i < r$ and $j \neq {\rm indeg}\,I_{\Delta}$. Therefore, we obtain 
$${\rm indeg}\,I_{\Delta} = \max\left\{ j \,:\, \beta_{i,i+j}(I_{\Delta}) \neq 0 \text{ for some } i \leq r - 1 \right\} = {\rm reg}_{\leq r-1}(I_{\Delta}),$$
which completes the proof.
\end{proof}

As an application of Theorem \ref{eq dual} and Proposition \ref{N_c,r}, we obtain the following corollary, which both unifies and extends the results of \cite[Corollary 0.3]{t} and \cite[Corollary 3.7]{y} in pure case.

\begin{cor}\label{gene}
Let $\Delta$ be a pure simplicial complex and let $r\geq 2$. Then, we have 
$${\rm reg}_{\leq r-1}I_{\Delta^{\vee}}-{\rm indeg}I_{\Delta^{\vee}}=\dim \Bbbk[\Delta]-S_{r}\mbox{-}{\rm depth}(\Bbbk[\Delta]).$$
\end{cor}
\begin{proof}
Notice that ${\rm reg}_{\leq r}\Bbbk[\Delta^{\vee}]={\rm reg}_{\leq r-1}I_{\Delta^{\vee}}-1$.  
Now, from Theorem \ref{eq dual}, we have ${\rm reg}_{\leq r-1}I_{\Delta^{\vee}}=n-S_{r}\mbox{-}{\rm depth}(\Bbbk[\Delta])$, where $[n]$ is the vertex set of $\Delta$. 
Hence, we obtain that 
$${\rm reg}_{\leq r-1}I_{\Delta^{\vee}}-{\rm indeg}I_{\Delta^{\vee}}=n-S_{r}\mbox{-}{\rm depth}(\Bbbk[\Delta])-{\rm ht}I_{\Delta}=\dim \Bbbk[\Delta]-S_{r}\mbox{-}{\rm depth}(\Bbbk[\Delta]),$$
as required. 
\end{proof}

\begin{cor}[\cite{t}, Corollary 0.3]
For a pure simplicial complex $\Delta$, we have 
$${\rm reg}I_{\Delta^{\vee}}-{\rm indeg}I_{\Delta^{\vee}}=\dim \Bbbk[\Delta]-{\rm depth}(\Bbbk[\Delta])$$
\end{cor}
\begin{proof}
Let $d=\dim \Bbbk[\Delta]$. 
Notice that $S_{d}$-{\rm depth}$(\Bbbk[\Delta])={\rm depth}(\Bbbk[\Delta])$. Also, we have ${\rm reg}_{\leq d-1}\Bbbk[\Delta^{\vee}]={\rm reg}\Bbbk[\Delta^{\vee}]$. Indeed, if ${\rm pd}I_{\Delta^{\vee}}<d$, then it is obvious. Suppose that ${\rm pd}I_{\Delta^{\vee}}=d$. Now, from Hochster's formula for graded Betti numbers, we have 
$$\beta_{d,d+j}(I_{\Delta^{\vee}})=\displaystyle  \sum_{\substack{F\in\Delta, \\[0.1cm] |F|=n-d-j}}{\rm dim}_{\Bbbk}\widetilde{H}_{d-1}({\rm link}_{\Delta}F;\Bbbk)\mbox{ for all }j.$$
Hence, $\beta_{{\rm pd}I_{\Delta^{\vee}}}(I_{\Delta^{\vee}})=\beta_{{\rm pd}I_{\Delta^{\vee}},{\rm pd}I_{\Delta^{\vee}}+{\rm indeg}I_{\Delta^{\vee}}}(I_{\Delta^{\vee}})$.  So, we have  ${\rm reg}_{\leq d-1}\Bbbk[\Delta^{\vee}]={\rm reg}\Bbbk[\Delta^{\vee}]$. Therefore, from Corollary \ref{gene}, we obtain the desired equality. 
\end{proof}

\begin{cor}[\cite{y}, Corollary 3.7]
For a pure simplicial complex $\Delta$ of codimension $c$ and $r\geq2$, the followings are equivalent: 
\begin{enumerate}
\item $\Delta$ satisfies Serre's condition $(S_{r})$,
\item $I_{\Delta^{\vee}}$ satisfies the condition $(N_{c,r})$. 
\end{enumerate}
\end{cor}
\begin{proof}
From Proposition \ref{N_c,r} and Corollary \ref{gene}, the assertion follows. 
\end{proof}

As a further consequence of Theorem \ref{eq dual}, we generalize a beautiful result proved by Smith on the depth of Stanley--Reisner rings to the case of the Serre depth. We begin by recalling Smith’s theorem. Let $\Delta^{i}=\{F\mid {\rm dim}F\leq i\}$ be the $i$-{\it skeleton} for $i=-1,\ldots,\dim \Delta$.

\begin{thm}[\cite{s1}, Theorem 3.7]\label{depth-CM}
For a simplicial complex $\Delta$, we have 
$${\rm depth}(\Bbbk[\Delta])=1+\max\{i\mid \Bbbk[\Delta^{i}]\mbox{ is Cohen--Macaulay}\}.$$
\end{thm}

We give a generalization result of Theorem \ref{depth-CM} as follows. 

\begin{thm}\label{formula}
Let $\Delta$ be a pure simplicial complex with $\dim \Delta=d-1$ and $2\leq r\leq d$. Then we have 
$$S_{r}\mbox{-}{\rm depth}(\Bbbk[\Delta])=S_{r}\mbox{-}{\rm depth}(\Bbbk[\Delta^{t-1}])\mbox{ for all }t\geq S_{r}\mbox{-}{\rm depth}(\Bbbk[\Delta]).$$
Moreover, we have 
$$S_{r}\mbox{-}{\rm depth}(\Bbbk[\Delta])=1+\max\{i\mid \Bbbk[\Delta^{i}]\mbox{ satisfies Serre's condition }(S_{r})\}.$$
\end{thm}
\begin{proof}
From Theorem \ref{eq dual}, there exists an integer $i\leq r$ such that $$\beta_{i, i+(n-S_{r}\mbox{-}{\rm depth}(\Bbbk[\Delta])-1)}(\Bbbk[\Delta^{\vee}])\neq 0.$$ Hence, from Hochster's formula on graded Betti numbers, there exists a face $F$ of $\Delta$ with $|F|=S_{r}\mbox{-}{\rm depth}(\Bbbk[\Delta])-i+1$ such that $\widetilde{H}_{i-2}({\rm link}_{\Delta}F;\Bbbk)\neq 0$. Then, since $|F|=S_{r}\mbox{-}{\rm depth}(\Bbbk[\Delta])-i+1\leq t-i+1\leq t+1$, $F\in\Delta^{t}$. Also, since $i-2=S_{r}\mbox{-}{\rm depth}(\Bbbk[\Delta])-|F|-1\leq t-|F|-1$, by the proof in \cite[Proposition 6.3.17]{v}, we have$$\widetilde{H}_{i-2}({\rm link}_{\Delta^{t}}F;\Bbbk)=\widetilde{H}_{i-2}({\rm link}_{\Delta}F;\Bbbk)\neq0.$$Thus, $\beta_{i-1,i+(n-S_{r}\mbox{-}{\rm depth}(\Bbbk[\Delta]))-1)}(\Bbbk[(\Delta^{t})^{\vee}])\neq 0$. Therefore, from Theorem \ref{eq dual}, we obtain $n-S_{r}\mbox{-}{\rm depth}(\Bbbk[\Delta])-1\leq n-S_{r}\mbox{-}{\rm depth}(\Bbbk[\Delta^{t}])-1$ and hence $S_{r}\mbox{-}{\rm depth}(\Bbbk[\Delta])\geq S_{r}\mbox{-}{\rm depth}(\Bbbk[\Delta^{t}])$. 
Conversely, from Theorem \ref{eq dual}, there exists an integer $i\leq r$ such that $$\beta_{i, i+(n-S_{r}\mbox{-}{\rm depth}(\Bbbk[\Delta^{t}]-1)}(\Bbbk[(\Delta^{t})^{\vee}])\neq 0.$$ From Hochster's formula on graded Betti numbers, there exists a face $F$ of $\Delta^{t}$ with $|F|=S_{r}\mbox{-}{\rm depth}(\Bbbk[\Delta^{t}])-i+1$ such that $\widetilde{H}_{i-2}({\rm link}_{\Delta^{t}}F;\Bbbk)\neq 0$. 
Also, since $t\geq S_{r}\mbox{-}{\rm depth}(\Bbbk[\Delta])\geq S_{r}\mbox{-}{\rm depth}(\Bbbk[\Delta^{t}])$, $i-2=S_{r}\mbox{-}{\rm depth}(\Bbbk[\Delta])-|F|-1\leq t-|F|-1$. Hence, by the proof in \cite[Proposition 6.3.17]{v}, we have$$\widetilde{H}_{i-2}({\rm link}_{\Delta}F;\Bbbk)=\widetilde{H}_{i-2}({\rm link}_{\Delta^{t}}F;\Bbbk)\neq0.$$Hence, $\beta_{i,i+(n-S_{r}\mbox{-}{\rm depth}(\Bbbk[\Delta^{t}])-1)}(\Bbbk[(\Delta)^{\vee}])\neq 0$. Finally we obtain that 
$n-S_{r}\mbox{-}{\rm depth}(\Bbbk[\Delta^{t}])-1\leq n-S_{r}\mbox{-}{\rm depth}(\Bbbk[\Delta])-1,$ and hence $S_{r}\mbox{-}{\rm depth}(\Bbbk[\Delta])\geq S_{r}\mbox{-}{\rm depth}(\Bbbk[\Delta^{t}])$. Therefore, we have $$S_{r}\mbox{-}{\rm depth}(\Bbbk[\Delta])=S_{r}\mbox{-}{\rm depth}(\Bbbk[\Delta^{t}])\mbox{ for all }t\geq S_{r}\mbox{-}{\rm depth}(\Bbbk[\Delta]).$$To prove the equality of the statement for $t-1$, we prove that $\Bbbk[\Delta^{S_{r}\mbox{-}{\rm depth}(\Bbbk[\Delta])-1}]$ satisfies Serre's condition $(S_{r})$. Fix an integer $-1\leq i\leq r-2$ and a face $F$ of $\Delta^{S_{r}\mbox{-}{\rm depth}(\Bbbk[\Delta])-1}$ such that $|F|\leq S_{r}\mbox{-}{\rm depth}(\Bbbk[\Delta])-i-2$. Then, since $i\leq S_{r}\mbox{-}{\rm depth}(\Bbbk[\Delta])-2-|F|$ and $|F|\leq S_{r}\mbox{-}{\rm depth}(\Bbbk[\Delta])-1$, we have$$\widetilde{H}_{i}({\rm link}_{\Delta^{S_{r}\mbox{-}{\rm depth}(\Bbbk[\Delta])-1}}F;\Bbbk)\simeq\widetilde{H}_{i}({\rm link}_{\Delta}F;\Bbbk).$$From  Lemma \ref{depth}, the latter group is zero, by \cite[page 4, following Theorem 1.7]{t} as desired. Finally we prove the equality$$S_{r}\mbox{-}{\rm depth}(\Bbbk[\Delta])=1+\max\{i\mid \Bbbk[\Delta^{i}]\mbox{ satisfies Serre's condition }(S_{r})\}.$$From the equality proved above, for all $t\geq S_{r}\mbox{-}{\rm depth}(\Bbbk[\Delta])$, we have 
$$\dim \Bbbk[\Delta^{t}]=t+1\geq S_{r}\mbox{-}{\rm depth}(\Bbbk[\Delta])+1=S_{r}\mbox{-}{\rm depth}(\Bbbk[\Delta^{t}])+1$$ and hence $\Delta^{t}$ does not satisfy $(S_{r})$. This means that $$S_{r}\mbox{-}{\rm depth}(\Bbbk[\Delta])\geq1+\max\{i\mid \Bbbk[\Delta^{i}]\mbox{ satisfies }(S_{r})\}.$$Also, the equality proved above, we have $\Bbbk[\Delta^{S_{r}\mbox{-}{\rm depth}(\Bbbk[\Delta])-1}]$ satisfies Serre's condition $(S_{r})$, as required. 
\end{proof}

Thanks to Theorem \ref{formula} and the proof technique in \cite[Lemma 2.2]{mmvv}, we have the following corollary. 
\begin{cor}
Let $I$ be a unmixed squarefree monomial ideal and let $f$ be a squarefree monomial. 
Then $S_{r}\mbox{-}{\rm depth}(S/(I:f))\geq S_{r}\mbox{-}{\rm depth}(S/I)$ for all $r\geq 2$. 
\end{cor}
\begin{proof}
Fix an integer  $r\geq 2$ and let $F={\rm supp}(f)$. We may assume that $f$ is a zero divisor otherwise $I:f=I$. Also, we may assume that $f$ is not in all minimal prime ideals of $I$ otherwise $I:f=S$. Let $\Delta$ and $\Delta^{\prime}$ be simplicial complexes such that $I_{\Delta}=I$ and $I_{\Delta^{\prime}}=I:f$. We assume that $\Delta^{i-1}$ satisfies Serre's condition $(S_{r})$. We suppose that $i>\dim \Delta^{\prime}$. Then, we take a facet $G$ of $\Delta^{\prime}$ containing $F$. Since $G$ is a face of $\Delta^{i-1}$ and this complex satisfies Serre's condition $(S_{r})$, in particular being pure (see \cite[Lemma 2.6]{mt}), $G$ is properly contained in a face of $\Delta$ of dimension $i-1$, which contradicts to the fact that $\dim \Delta\geq\dim \Delta^{\prime}$. Therefore, we get $i-1\leq\dim \Delta^{\prime}$. 
Notice that $$(\Delta^{\prime})^{i-1}=({\rm star}_{\Delta}(F))^{i-1}={\rm star}_{\Delta^{i-1}}(F),$$
where ${\rm star}_{\Delta}F=\{G\in\Delta\,\,: G\cup F\in\Delta\}$. Now, since the star of a face of a simplicial complex satisfying Serre's condition $(S_{r})$ also satisfies Serre's condition $(S_{r})$, $(\Delta^{\prime})^{i-1}$ satisfies Serre's condition $(S_{r})$. Therefore, by Theorem \ref{formula}, we have 
$$S_{r}\mbox{-}{\rm depth}(S/(I:f))\geq S_{r}\mbox{-}{\rm depth}(S/I),$$
which completes the proof. 
\end{proof}

Moreover, by Theorem \ref{formula}, we obtain the Serre depth on the link of a simplicial complex. This corresponds to the Serre depth associated with the localization of Stanley–Reisner rings at prime ideals.

\begin{cor}
Let $\Delta$ be a simplicial complex and $r\geq2$. Then we have 
$$S_{r}\mbox{-}{\rm depth}(\Bbbk[\Delta])\leq S_{r}\mbox{-}{\rm depth}(\Bbbk[{\rm link}_{\Delta}F])+|F|\mbox{ for all }F\in\Delta.$$
\end{cor}
\begin{proof}
Fix a face $F\in\Delta$. If $S_{r}\mbox{-}{\rm depth}(\Bbbk[\Delta])\leq|F|$, then the desired inequality clearly holds. Hence, we may assume that $S_{r}\mbox{-}{\rm depth}(\Bbbk[\Delta])>|F|$. By the proof in \cite[Proposition 6.3.17]{v} and the assumption that $S_{r}\mbox{-}{\rm depth}(\Bbbk[\Delta])-1\geq|F|$, we have$${\rm link}_{\Delta^{S_{r}\mbox{-}{\rm depth}(\Bbbk[\Delta])-1}}F=({\rm link}_{\Delta}F)^{S_{r}\mbox{-}{\rm depth}(\Bbbk[\Delta])-1-|F|}.$$
Now, from Theorem \ref{formula} and the fact that the link of every face satisfies $(S_{r})$ if $\Delta$ satisfies Serre's condition $(S_{r})$, $({\rm link}_{\Delta}F)^{S_{r}\mbox{-}{\rm depth}(\Bbbk[\Delta])-1-|F|}$ satisfies Serre's condition $(S_{r})$. Hence, we have 
$$S_{r}\mbox{-}{\rm depth}(\Bbbk[\Delta])-1-|F|\leq\max\{i\,\,:({\rm link}_{\Delta}F)^{i}\mbox{ satisfies Serre's condition }(S_{r})\}.$$Therefore, again by Theorem \ref{formula}, we obtain that $$S_{r}\mbox{-}{\rm depth}(\Bbbk[\Delta])-|F|\leq S_{r}\mbox{-}{\rm depth}(\Bbbk[{\rm link}_{\Delta}F]),$$
which implies the desired inequality. 
\end{proof}

In the rest of this section, we consider an operation to Stanley--Reisner rings called an 1-vertex inflation, which was introduced in \cite{bh1}. We follow the notation in \cite{rt}. We begin recalling the definition of it. Let $\Delta$ be a simplicial complex on $[n]$. We write $I_{\Delta}=(m_{1},\ldots, m_{s})$, where $m_{1},\ldots, m_{k}$ are monomials divisible by $x_{n}$ and $m_{k+1},\ldots, m_{s}$ are monomials indivisible $x_{n}$. Let $J$ be the squarefree monomial ideal in $S^{\prime}=\Bbbk[x_{1},\ldots, x_{n}, x_{n+1}]$ such that 
$$J=(m_{1}x_{n+1},\ldots, m_{k}x_{n+1},m_{k+1},\ldots, m_{s}).$$
Let $\Delta^{\prime}$ be the simplicial  complex on $[n+1]$ such that $J=I_{\Delta^{\prime}}$. This simplicial complex is called the {\it 1-vertex inflation} of $\Delta$ with respect to $x_{n}$.  

To study the Serre depth on a 1-vertex inflation, we prepare the following lemma concerning the relationship between the Alexander dual complexes of a simplicial complex and of its 1-vertex inflation.
\begin{lemma}\label{lemma for 1-vertex}
Let $\Delta$ be a simplicial complex on $X_{[n]}=\{x_{1},\ldots, x_{n}\}$ and $\Delta^{\prime}$ be the 1-vertex inflation of $\Delta$ with respect to $x_{n}$,  $W^{\prime}\subset X_{n}\cup\{x_{n+1}\}$. Then the following statements hold:
\begin{enumerate}
\item If $x_{n+1}\notin W^{\prime}$, then $(\Delta^{\prime})^{\vee}_{W^{\prime}}=\Delta^{\vee}_{W^{\prime}}$.  \\
\item If $x_{n}\notin W^{\prime}$ and $x_{n+1}\in W^{\prime}$, then $(\Delta^{\prime})^{\vee}_{W^{\prime}}\simeq\Delta^{\vee}_{(W^{\prime}\setminus\{x_{n+1}\})\cup\{x_{n}\}}$. \\
\item If $x_{n}, x_{n+1}\in W^{\prime}$, then the geometric realizations $|(\Delta^{\prime})^{\vee}_{W^{\prime}}|$ and $|\Delta^{\vee}_{W^{\prime}\setminus\{x_{n+1}\}}|$ are homotopy equivalent. 
\end{enumerate}
\end{lemma}
\begin{proof}
A straightforward computation shows that
$$(\Delta^{\prime})^{\vee}=\langle\{F\cup\{x_{n+1}\}\,\,: x_{n}\in F\in\mathcal{F}(\Delta^{\vee})\}\cup\{G\,\,: x_{n}\notin G\in\mathcal{F}(\Delta^{\vee})\}\rangle.$$Therefore, it is clear that $(1)$ and $(2)$ hold. 
To show the claim of $(3)$, we now take two continuous maps
$$f:|\Delta^{\vee}_{W^{\prime}\setminus\{x_{n+1}\}}|\rightarrow|(\Delta^{\prime})^{\vee}_{W^{\prime}}|\mbox{ and }g: |(\Delta^{\prime})^{\vee}_{W^{\prime}}|\rightarrow|\Delta^{\vee}_{W^{\prime}\setminus\{x_{n+1}\}}|$$
such that $f(x_{i})=x_{i}, g(x_{i})=x_{i}$ for all $1\leq i\leq n$ and $g(x_{n+1})=g(x_{n})$. Notice that $g\circ f$ is the identity on $|(\Delta^{\prime})^{\vee}_{W^{\prime}}|$. We now define the continuous map $h:|(\Delta^{\prime})^{\vee}_{W^{\prime}}|\times[0,1]\rightarrow|(\Delta^{\prime})^{\vee}_{W^{\prime}}|$ by 
$$h\left(\left(\displaystyle\sum_{i=1}^{n}a_{i}x_{i}\right)+a_{n+1}x_{n+1}, t\right)=\left(\displaystyle\sum_{i=1}^{n}a_{i}x_{i}\right)+a_{n+1}(tx_{n+1}+(1-t)x_{n}).$$ Then, by the definition of $h$, $h(-,1)$ is the identity on $|(\Delta^{\prime})^{\vee}_{W^{\prime}}|$ and $h(-,0)=f\circ g$. Hence, $f\circ g$ is homotopic to the identity on $|(\Delta^{\prime})^{\vee}_{W^{\prime}}|$, as required. 
\end{proof}

\begin{thm}\label{1-vertex}
Let $\Delta$ be a pure simplicial complex on $X_{[n]}=\{x_{1},\ldots, x_{n}\}$ and $\Delta^{\prime}$ be the 1-vertex inflation of $\Delta$ with respect to $x_{n}$ and let $r\geq 2$. Then, we have 
$$S_{r}\mbox{-}{\rm depth}(S^{\prime}/I_{\Delta^{\prime}})=S_{r}\mbox{-}{\rm depth}(S/I_{\Delta})+1.$$
Therefore, $\Delta$ satisfies Serre's condition $(S_{r})$ if and only if $\Delta^{\prime}$ does. 
\end{thm}
\begin{proof}
Notice that $\beta_{i,j}(S/I_{\Delta^{\vee}})\neq 0$ implies that $\beta_{i,j}(S/I_{(\Delta^{\prime})^{\vee}}|_{W})\neq 0$. Indeed, for a subset $W$ of $[n]$ such that $\dim_{\Bbbk}\widetilde{H}_{|W|-i-1}(\Delta^{\vee}_{W};\Bbbk)\neq 0$, then we have $$\dim_{\Bbbk}\widetilde{H}_{|W|-i-1}((\Delta^{\prime})^{\vee}_{W};\Bbbk)=\dim_{\Bbbk}\widetilde{H}_{|W|-i-1}(\Delta^{\vee}|_{W};\Bbbk)\neq 0.$$ The implication follows from Hochster’s formula on graded Betti numbers.
Therefore, $S_{r}\mbox{-}{\rm depth}(S^{\prime}/I_{\Delta^{\prime}})\leq S_{r}\mbox{-}{\rm depth}(S/I_{\Delta})+1$ follows from Theorem \ref{eq dual}. We now prove the reverse inequality. Let $s^{\prime}=S_{r}\mbox{-}{\rm depth}(S^{\prime}/I_{\Delta^{\prime}})$. Then, there exists $i\leq r$ such that $\beta_{i,i+n-s^{\prime}}(S^{\prime}/I_{(\Delta^{\prime})^{\vee}})\neq0$ from Theorem \ref{eq dual}. From Hochster's formula on graded Betti numbers, there exists a subset $W^{\prime}$ of $V(\Delta^{\prime})$ such that $|W^{\prime}|=i+n-s^{\prime}$ and $\dim_{\Bbbk}\widetilde{H}_{n-s^{\prime}-1}((\Delta^{\prime})^{\vee}_{W^{\prime}};\Bbbk)\neq 0$. We distinguish the following cases:

\vspace{0.1cm}
{\it Case 1}: $x_{n+1}\notin W^{\prime}$

\vspace{0.1cm} From Lemma \ref{lemma for 1-vertex} (1), we have $(\Delta^{\prime})^{\vee}_{W^{\prime}}=\Delta^{\vee}_{W^{\prime}}$. Hence, we obtain 
$$\dim_{\Bbbk}\widetilde{H}_{n-s^{\prime}-1}(\Delta^{\vee}_{W^{\prime}};\Bbbk)=\dim_{\Bbbk}\widetilde{H}_{n-s^{\prime}-1}((\Delta^{\prime})^{\vee}_{W^{\prime}};\Bbbk)\neq 0.$$ 
Thus, from Hochster's formula on graded Betti numbers, we have $\beta_{i,i+n-s^{\prime}}(S/I_{\Delta^{\vee}})\neq0$. 

\vspace{0.1cm}
{\it Case 2}: $x_{n}\notin W^{\prime}$ and $x_{n+1}\in W^{\prime}$

\vspace{0.1cm}From Lemma \ref{lemma for 1-vertex} (2), we have $(\Delta^{\prime})^{\vee}_{W^{\prime}}\simeq\Delta^{\vee}_{(W^{\prime}\setminus\{x_{n+1}\})\cup\{x_{n}\}}$. Hence, we obtain 
$$\dim_{\Bbbk}\widetilde{H}_{n-s^{\prime}-1}(\Delta^{\vee}_{(W^{\prime}\setminus\{x_{n+1}\})\cup\{x_{n}\}};\Bbbk)=\dim_{\Bbbk}\widetilde{H}_{n-s^{\prime}-1}((\Delta^{\prime})^{\vee}_{W^{\prime}};\Bbbk)\neq 0.$$ 
Thus, from Hochster's formula on graded Betti numbers, we have $\beta_{i,i+n-s^{\prime}}(S/I_{\Delta^{\vee}})\neq0$. 

\vspace{0.1cm}
{\it Case 3}: $x_{n}, x_{n+1}\in W^{\prime}$

\vspace{0.1cm}From Lemma \ref{lemma for 1-vertex} (3), the geometric realizations $|(\Delta^{\prime})^{\vee}_{W^{\prime}}|$ and $|\Delta^{\vee}_{W^{\prime}\setminus\{x_{n+1}\}}|$ are homotopy equivalent, that is, $$\widetilde{H}_{k}((\Delta^{\prime})^{\vee}_{W^{\prime}};\Bbbk)\simeq\widetilde{H}_{k}(\Delta^{\vee}_{W^{\prime}\setminus\{x_{n+1}\}};\Bbbk)\mbox{ for all }k.$$
Hence, we obtain $$\dim_{\Bbbk}\widetilde{H}_{n-s^{\prime}-1}(\Delta^{\vee}_{W^{\prime}\setminus\{x_{n+1}\}};\Bbbk)=\dim_{\Bbbk}\widetilde{H}_{n-s^{\prime}-1}((\Delta^{\prime})^{\vee}_{W^{\prime}};\Bbbk)\neq 0.$$ 
Since $|W^{\prime}\setminus\{x_{n+1}\}|=|W^{\prime}|-1$, from Hochster's formula on graded Betti numbers, we have $\beta_{i-1, i-1+n-s^{\prime}}(S/I_{\Delta^{\vee}})\neq0$. In all cases, the assertion follows from Theorem \ref{eq dual}. Moreover, since $\dim \Bbbk[\Delta^{\prime}]=\dim \Bbbk[\Delta]+1$, the latter one holds. 
\end{proof}

As a corollary, we derive a result on the Serre depth of the cover ideal of the graph obtained by parallelization. Let us recall the definition of parallelization (see e.g. \cite{mmv}).  
Let $G$ be a graph on the vertex set $[n]$, and let ${\bf a} = (a_1, \ldots, a_n) \in \mathbb{Z}_{> 0}^n$. The {\it parallelization} $G^{\bf a}$ is the graph with vertex set
$$V(G^{\bf a}) = \{1_1, \ldots, 1_{a_1}, \ldots, n_1, \ldots, n_{a_n}\}$$and edge set$$E(G^{\bf a}) = \left\{ \{i_r, j_s\} \mid \{i, j\} \in E(G),\ 1 \le r \le a_i,\ 1 \le s \le a_j \right\}.
$$
Then, by Theorem \ref{1-vertex}, we obtain the following corollary. 
\begin{cor}
Let $G$ be a graph on the vertex set $[n]$ and let ${\bf a}\in\mathbb{Z}_{>0}^{n}$, $r\geq 2$. Then, 
$$S_{r}\mbox{-}{\rm depth}(S^{\prime}/J(G^{\bf a}))=S_{r}\mbox{-}{\rm depth}(S/J(G))+|{\bf a}|-n,$$
where $|{\bf a}|=a_{1}+\cdots+a_{n}$ and $S^{\prime}=\Bbbk[V(G^{\bf a})]$
\end{cor}
\begin{proof}
Assume that $a_{1}\geq 2$ and fix $r\geq 2$.  
Let $G^{\prime}$ be the induced subgraph of $G^{\bf a}$ on $\{1_{1},1_{2},2\ldots, n\}.$ Then, for a minimal vertex cover $C$ of $G$ containing $1$, $(C\cup\{1_{1}, 1_{2}\})\setminus\{1\}$ is a minimal vertex cover of $G^{\prime}$ by the definition of parallelization. Hence we have $$J(G^{\prime})=x_{1_{2}}(m\,\,: m\in\mathcal{G}(J(G))\mbox{ with }x_{1}\mid m)+(m\,\,: m\in\mathcal{G}(J(G))\mbox{ with }x_{1}\nmid m),$$
where we identify $x_{1_{1}}$ with $x_{1}$. Therefore $\Delta(G^{\prime})^{\vee}$ is the 1-vertex inflation of $\Delta(G)^{\vee}$ with respect to $x_{1}$. From Theorem \ref{1-vertex}, we get $$S_{r}\mbox{-}{\rm depth}(\Bbbk[V(G^{\prime})]/J(G^{\prime}))=S_{r}\mbox{-}{\rm depth}(S/J(G))+1.$$ By repeating this process inductively, we obtain the desired equality.
\end{proof}


\section{The Serre depth and the depth on the symbolic powers}\label{mono}
In this section, we treat monomial ideals and symbolic powers of squarefree monomial ideals. First, we consider monomial ideals. It is easy to see that the following result by using \cite[Corollary 2.3]{htt}. This is a generalized result of Serre's condition $(S_{r})$ property of a monomial ideal is inherited by its radical. 
\begin{prop}\label{polarization}
For a monomial ideal $I$ and $r\geq 2$
$$S_{r}\mbox{-}{\rm depth}(S/I)\leq S_{r}\mbox{-}{\rm depth}(S/\sqrt{I}).$$
\end{prop}

In order to investigate the Serre depth of monomial ideals, we consider a concept called polarization. Let us recall the definition of the polarization, following the notation in \cite{mmvv}.  
Let $S = \Bbbk[x_1, \ldots, x_n]$ be a polynomial ring, and let $I$ be a monomial ideal of $S$. Denote by $\mathcal{G}(I)$ the set of minimal monomial generators of $I$ and set $\mathcal{G}(I)=\{m_{1},\ldots, m_{s}\}$. For each $i$, define$$\gamma_i = \max\{ \deg_{x_i}(u) \mid u \in \mathcal{G}(I) \},$$where $\deg_{x_i}(u)$ denotes the exponent of $x_i$ in the monomial $u$. We define a new set of variables by$$X_{I} = \bigcup_{1 \leq i \leq n} \{ x_{i,2}, \ldots, x_{i,\gamma_i} \},$$where the set $\{ x_{i,2}, \ldots, x_{i,\gamma_i} \}$ is taken to be empty if $\gamma_i = 0$ or $\gamma_i = 1$.
Observe that a power $x_i^{c_i}$ of a variable $x_i$, $1 \le c_i \le \gamma_i$, polarizes as follows:
\[
(x_i^{c_i})^{\mathrm{pol}} = 
\begin{cases} 
x_i & \text{if } \gamma_i = 1,\\[1mm]
x_{i,2} \cdots x_{i,c_i+1} & \text{if } c_i < \gamma_i,\\[1mm]
x_{i,2} \cdots x_{i,\gamma_i} x_i & \text{if } c_i = \gamma_i.
\end{cases}
\]
This defines a polarization $m_{i}^{\rm pol}$ of each $m_{i}$ for $i = 1, \dots, s$. 
The {\it polarization} $I^{\rm pol}$ of $I$ is then the ideal in $S[X_{I}]$ generated by $m_{1}^{\rm pol}, \dots, m_{s}^{\rm pol}$.

\begin{prop}\label{polarization}
For a monomial ideal $I$ of $S$ and $r\geq 2$, we have 
$$S_{r}\mbox{-}{\rm depth}(S^{\rm pol}/I^{\rm pol})=S_{r}\mbox{-}{\rm depth}(S/I)+|X_{I}|.$$
\end{prop}
\begin{proof}
By using results \cite[Theorem 4.9, Corollary 4.10]{s2}, one can see that $$\dim {\rm Ext}_{S}^{j}(S/I,\omega_{S})+|X_{I}|=\dim {\rm Ext}_{S^{\rm pol}}^{j}(S^{\rm pol}/I^{\rm pol},\omega_{S^{\rm pol}}),$$
if $\dim {\rm Ext}_{S}^{j}(S/I,\omega_{S})>0$. 
From now, we prove the equality of the statement. Let $p=S_{r}\mbox{-}{\rm depth}(S^{\rm pol}/I^{\rm pol})$. Then, by the above equality and the definition of the Serre depth, we have $$\dim {\rm Ext}_{S}^{n-(-|X_{I}|+p)}(S/I,\omega_{S})+|X_{I}|=\dim {\rm Ext}_{S^{\rm pol}}^{n+|X_{I}|-p}(S^{\rm pol}/I^{\rm pol},\omega_{S^{\rm pol}})\geq p-r+1,$$and hence,$$\dim {\rm Ext}_{S}^{n-(|X_{I}|+p)}(S/I,\omega_{S})\geq(-|X_{I}|+p)-r+1.$$ Therefore, $p\geq S_{r}\mbox{-}{\rm depth}(S/I)+|X_{I}|.$ Conversely, let $q=S_{r}\mbox{-}{\rm depth}(S/I)$. Then, by the above equality and the definition of the Serre depth, we have $$\dim {\rm Ext}_{S^{\rm pol}}^{n+|X_{I}|-(|X_{I}+q|)}(S^{\rm pol}/I^{\rm pol},\omega_{S^{\rm pol}})=\dim {\rm Ext}_{S}^{n-q}(S/I,\omega_{S})+|X_{I}|\geq |X_{I}|+q-r+1,$$as required. 
\end{proof}

The following corollary is well-known (see e.g. \cite[Proof of Theorem 4.1]{mt2}):
\begin{cor}
For a monomial ideal $I$ of $S$ and $r\geq 2$, $S/I$ satisfies Serre's condition $(S_{r})$ if and only if $S^{\rm pol}/I^{\rm pol}$ does. 
\end{cor}

By Proposition \ref{lin} and Proposition \ref{polarization}, we obtain the following corollary: 
\begin{prop}
Let $I$ be a monomial ideal generated in one degree $d$. If $I$ has a d-linear resolution and $S/I$ satisfies Serre's condition $S_{r}$ for some $r>d$, then $S/I$ is Cohen--Macaulay. 
\end{prop}
\begin{proof}
In this case, by Proposition \ref{lin} and Proposition \ref{polarization}, we have $\dim S/I=S_{r}\mbox{-}{\rm depth}(S/I)={\rm depth}(S/I)$, as required. 
\end{proof}

In the rest of this section, we consider about the symbolic powers of Stanley--Reisner ideals. It is known that the reformulate version of Takayama's formula on symbolic powers of Stanley--Reisner ideals by Kataoka and the authors of this paper. For ${\bf a}=(a_1, \ldots, a_n) \in\mathbb{Z}^n$, we set $x^{\bf a}=x_{1}^{a_1}\cdots x_{n}^{a_n}$ and ${\rm supp}_{+}{\bf a}=\{i \,:\, \ a_i>0\}$,  ${\rm supp}_{-}{\bf a}=\{i \,:\, \ a_i<0\}$. For $F\subset [n]$, we define $S_F=S[x_j^{-1} \,:\, j\in F]$. Then the simplicial complex $\Delta_{\bf a}(I)$ is defined as 
$$\Delta_{\bf a}(I)=\{F\setminus {\rm supp}_{-}{\bf a}; \ {\rm supp}_{-}{\bf a} \subset F, x^{\bf a}    \not\in IS_F \}.$$The following Theorem is the reformulate version of Takayama's formula on the symbolic powers of Stanley--Reisner ideals. 

\begin{thm}[\cite{kmt}, Theorem 2.2] \label{Takayama}
Let $\Delta$ be a  simplicial complex.  Then the Hilbert series of the local cohomology modules of $S/I_{\Delta}^{(\ell)}$ is given by 
$$F\big(H^{i}_{\frak{m}}(S/I_{\Delta}^{(\ell)}),t\big)=\sum_{F \in \Delta} \left(\sum_{{\bf a}\in \mathbb{N}^{V({\rm link}_{\Delta}F)}} \dim_k \widetilde{H}_{i-|F|-1}\big(\Delta_{{\bf a}}(I_{{\rm link}_ {\Delta}F}^{(\ell)});k\big)t^{|{\bf a}|}\left(\frac{t^{-1}}{1-t^{-1}}\right)^{|F|}\right),$$where $ \mathbb{N}^{V({\rm link}_{\Delta}F)}\hookrightarrow  \mathbb{N}^{V(\Delta)}$ is induced by $ V({\rm link}_{\Delta}F)\hookrightarrow  V(\Delta)$. 
\end{thm}

\begin{cor}[\cite{kmt}, Corollary 2.3] \label{dual}
For a simplicial complex $\Delta$, we have 
\begin{eqnarray*}
F\big(K^{i}_{S/I_{\Delta}^{(\ell)}},t\big)
& = & \sum_j \dim _{k}\big[K^{i}_{S/I_{\Delta}^{(\ell)}}\big]_{j}\ t^{j}\\[-0.2cm]
& = & \sum_{F \in \Delta} \big(\sum_{{\bf a}\in \mathbb{N}^{V({\rm link}_{\Delta}F)}} \dim_k \widetilde{H}_{i-|F|-1}\big(\Delta_{{\bf a}}(I_{{\rm link}_ {\Delta}F}^{(\ell)});k\big)t^{-|{\bf a}|}\left(\frac{t}{1-t}\right)^{|F|}\big),
\end{eqnarray*}
where $ \mathbb{N}^{V({\rm link}_{\Delta}F)}\hookrightarrow  \mathbb{N}^{V(\Delta)}$ is induced by $ V({\rm link}_{\Delta}F)\hookrightarrow  V(\Delta)$. 
\end{cor}

In order to generalize \cite[Corollary 3.3]{rty} to the Serre depth 
$S_2\mbox{-}{\rm depth}(S/I_{\Delta}^{(2)})$, we first prepare the following lemma.

\begin{lemma} \label{global}
Let $\Delta$ be a pure simplicial complex with $\dim \Delta \ge 1$. 
Then the following conditions are equivalent$:$
\begin{enumerate}
\item $\widetilde{H}_{0}\big(\Delta_{{\bf a}}(I_{\Delta}^{(2)});k\big)=0$ for any ${\bf a}\in \mathbb{N}^{n}$.
\item ${\rm depth}(S/I_{\Delta}^{(2)}) \ge 2$. 
\item ${\rm diam} \Delta^{(1)} \le 2$. 
\end{enumerate}
\end{lemma}
\begin{proof}
Since $\Delta$ is pure, $S/I_{\Delta}^{(2)}$ is unmixed. Hence the equivalence of (1) and (2) follows from  Theorem \ref{Takayama}. Also, the equivalence of (2) and (3) is proved in \cite[Theorem 3.2]{rty}.
\end{proof}

\begin{thm}  \label{S2depth}
Let $\Delta$ be a pure simplicial complex with $ \dim \Delta=d-1$ and $2 \le s\le d$.
Then the following conditions are equivalent$:$
\begin{enumerate}
\item $\mbox{\rm $S_2$-depth}(S/I_{\Delta}^{(2)})\ge s$.
\item ${\rm diam} (({\rm link}_{\Delta} F)^{(1)}) \le 2$ for any face $F \in \Delta$ with 
$\dim {\rm link}_{\Delta} F \ge d-s+1$. 
\end{enumerate}
\end{thm}
\begin{proof}
Let $F$ be a face of $\Delta$ with $\dim {\rm link}_{\Delta} F \ge d-s+1$. Then, by the assumption and $|F|+1 \le s-1$, we have $\dim K^{|F|+1}_{S/I_{\Delta}^{(2)}} < |F|$. Hence, by Corollary \ref{dual},
$ \widetilde{H}_{0}\big(\Delta_{{\bf a}}(I_{{\rm link}_ {\Delta}F}^{(2)});k\big)=0$  for any ${\bf a}\in \mathbb{N}$. Thus, by Lemma \ref{global},  ${\rm diam} ({{\rm link}_ {\Delta}F})^{(1)} \leq 2$. Conversely, by Lemma \ref{global}, $ \widetilde{H}_{0}\big(\Delta_{{\bf a}}(I_{{\rm link}_ {\Delta}F}^{(2)});k\big)=0$  for any ${\bf a}\in \mathbb{N}^{n}$. Then Corollary \ref{dual} and $|F|+1 \le s-1$ implies that $\dim K^{|F|+1}_{S/I_{\Delta}^{(2)}} < |F|$   for any $F \in \Delta$ with $\dim {\rm link}_{\Delta} F \ge d-s+1$. Hence we have   $\mbox{\rm $S_2$-depth}(S/I_{\Delta}^{(2)})\ge s$.
\end{proof}

We provide the following characterizations of the Serre depth for $(S_2)$ on the symbolic powers of a Stanley--Reisner ideal, based on Theorem \ref{Takayama}:

\begin{prop}\label{S_2-depth comb}
Let $\Delta$ be a pure simplicial complex and $\ell\geq1$. Then, we have 
\[
S_{2}\mbox{-}{\rm depth}(S/I_{\Delta}^{(\ell)})
= \min \Bigl\{\, |F| :
   \substack{
     \displaystyle F\in\Delta\text{ and there exists a vector } {\bf a}\in\mathbb{N}^{V({\rm link}_{\Delta}F)}\\
     \displaystyle \text{such that } \Delta_{\bf a}(I_{{\rm link}_{\Delta}F}^{(\ell)}) \text{ is disconnected}
   } \,\Bigr\} + 1.
\]
\end{prop}
\begin{proof}
First, we prove that the left-hand side is less than or equal to the right-hand side. Take $F\in\Delta$ such that $\Delta_{\bf a}(I_{{\rm link}_{\Delta}F}^{(\ell)})$ is disconnected for some ${\bf a}\in\mathbb{N}^{V({\rm link}_{\Delta}F)}$. It follows that $\widetilde{H}_{(|F|+1)-|F|-1}(\Delta_{\bf a}(I_{{\rm link}_{\Delta}F}^{(\ell)}); \Bbbk)=\widetilde{H}_{0}(\Delta_{\bf a}(I_{{\rm link}_{\Delta}F}^{(\ell)}); \Bbbk)\neq0$. Hence by the Hochster's formula for local cohomology modules, we obtain that $\dim K_{S/I_{\Delta}^{(\ell)}}^{|F|+1}\geq|F|$ and thus, we obtain that $|F|+1\geq S_{2}\mbox{-}{\rm depth}(S/I_{\Delta}^{(\ell)})$. Conversely, by the definition and Theorem \ref{Takayama} of the Serre depth for $(S_{2})$, there exist a face $F\in\Delta$ and ${\bf a}\in\mathbb{N}^{V({\rm link}_{\Delta}F)}$ with $|F|\geq S_{2}\mbox{-}{\rm depth}(S/I_{\Delta}^{(\ell)})-1$ such that $\dim_{\Bbbk}\widetilde{H}_{S_{2}\mbox{-}{\rm depth}(S/I_{\Delta}^{(\ell)})-|F|-1}(\Delta_{\bf a}(I_{{\rm link}_{\Delta}F}^{(\ell)});\Bbbk)\neq0$. Hence we have $\widetilde{H}_{0}(\Delta_{\bf a}(I_{{\rm link}_{\Delta}F}^{(\ell)});\Bbbk)\neq0$, that is, $\Delta_{\bf a}(I_{{\rm link}_{\Delta}F}^{(\ell)})$ is disconnected. 
\end{proof}

\begin{prop}\label{S_2-depth alg}
Let $\Delta$ be a pure simplicial complex and $\ell\geq1$. Then we have $$S_{2}\mbox{-}{\rm depth}(S/I_{\Delta}^{(\ell)})=\min\{|F|\,\,: F\in\Delta\mbox{ with }H_{\mathfrak{m}}^{1}(\Bbbk[V({\rm link}_{\Delta}F)]/I_{{\rm link}_{\Delta}F}^{(\ell)})\neq0\}+1$$
\end{prop}
\begin{proof}
For a face $F$ of $\Delta$, by Theorem \ref{Takayama}, there exists a vector ${\bf a}\in\mathbb{N}^{{\rm link}_{\Delta}F}$ such that $\Delta_{\bf a}(I_{{\rm link}_{\Delta}F}^{(\ell)})$ is disconnected if and only if $H_{\mathfrak{m}}^{1}(\Bbbk[V({\rm link}_{\Delta}F)]/I_{{\rm link}_{\Delta}F}^{(\ell)})\neq0$. Hence, by Proposition \ref{S_2-depth comb}, we obtain that the desired equality.  
\end{proof}

As a corollary, we prove the following statement. 
\begin{cor}
For a pure simplicial complex $\Delta$, we have 
$$S_{2}\mbox{-}{\rm depth}(S/I_{\Delta}^{(2)})\geq S_{2}\mbox{-}{\rm depth}(S/I_{\Delta}^{(\ell)})\mbox{ for all }\ell\geq3.$$
\end{cor}
\begin{proof}
Fix $\ell\geq3$. By \cite[Theorem 2.7]{nt}, one can see that ${\rm depth}(S/I_{\Delta}^{(2)})\geq{\rm depth}(S/I_{\Delta}^{(\ell)})$. This implies that if $H_{\mathfrak{m}}^{1}(S/I_{\Delta}^{(2)})\neq0$, then $H_{\mathfrak{m}}^{1}(S/I_{\Delta}^{(\ell)})\neq0$. Therefore, for any face $F$ of $\Delta$, if $H_{\mathfrak{m}}^{1}(\Bbbk[V({\rm link}_{\Delta}F)]/I_{{\rm link}_{\Delta}F}^{(2)})\neq0$, then $H_{\mathfrak{m}}^{1}(\Bbbk[V({\rm link}_{\Delta}F)]/I_{{\rm link}_{\Delta}F}^{(\ell)})\neq0$, which implies the desired equality holds by Proposition \ref{S_2-depth alg}. 
\end{proof}

Let us recall the following lemma, which will be used to analyze the complexes appearing in Takayama's formula and in Theorem \ref{Takayama}, proved by Minh and Trung:

\begin{lemma}[\cite{mt}]\label{symbolic power cpx}
For a simplicial complex $\Delta$ on the vertex set $[n]$, a positive integer $\ell$, and ${\bf a}\in\mathbb{N}^{n}$, we have 
$$\Delta_{\bf a}(I_{\Delta}^{(\ell)})=\langle F\in\mathcal{F}(\Delta)\,\,: \displaystyle\sum_{i\notin F}a_{i}\leq\ell-1\rangle.$$
\end{lemma}

Using this characterization of $\Delta_{\bf a}(I_{\Delta}^{(\ell)})$, we provide a combinatorial necessary and sufficient condition for the vanishing of the first local cohomology modules of symbolic powers of the Stanley–Reisner ideal: 

\begin{thm}\label{1st local coho}
Let $\Delta$ be a simplicial complex on $[n]$ with $\dim \Delta\geq1$ and let $\ell\geq2$. Then the following conditions are equivalent:
\begin{enumerate}
\item $H_{\mathfrak{m}}^{1}(S/I_{\Delta}^{(\ell)})=0$ 
\item For any $F_{0}, G_{0}\in\mathcal{F}(\Delta)$ with $F_{0}\cap G_{0}=\emptyset$ and any ${\bf a}=(a_{1},\ldots, a_{n})\in\mathbb{N}^{n}$ such that $\displaystyle\sum_{i\in[n]}a_{i}\geq\ell$, $\displaystyle\sum_{i\notin F_{0}}a_{i}=\ell-1$ and $\displaystyle\sum_{i\notin G_{0}}a_{i}=\ell-1$, there exist facets $F, G\in\mathcal{F}(\Delta_{\bf a}(I_{\Delta}^{(\ell)}))$ such that $F\cap F_{0}\neq\emptyset$, $G\cap G_{0}\neq\emptyset$ and $F\cap G\neq\emptyset$. 
\end{enumerate}
\end{thm}
\begin{proof}We set $\Delta_{\bf a}=\Delta_{\bf a}(I_{\Delta}^{(\ell)})$ for all ${\bf a}\in\mathbb{N}^{n}$. Notice that $H_{\mathfrak{m}}^{1}(S/I_{\Delta}^{(\ell)}) = 0$ if and only if, for all ${\bf a} \in \mathbb{N}^n$, the complex $\Delta_{\bf a}$ is connected, by Theorem \ref{Takayama}. We suppose that $H_{\mathfrak{m}}^{1}(S/I_{\Delta}^{(\ell)})=0$. Fix $F_{0}, G_{0}\in\mathcal{F}(\Delta)$ with $F_{0}\cap G_{0}=\emptyset$ and any ${\bf a}=(a_{1},\ldots, a_{n})\in\mathbb{N}^{n}$ such that $\displaystyle\sum_{i\in[n]}a_{i}\geq\ell$, $\displaystyle\sum_{i\notin F_{0}}a_{i}=\ell-1$ and $\displaystyle\sum_{i\notin G_{0}}a_{i}=\ell-1$. We distinguish the following cases:

\vspace{0.1cm}
{\it Case 1}: $\displaystyle\sum_{i\in[n]}a_{i}=2\ell-2$. 

\vspace{0.1cm}In this case, we prove that one of the following conditions holds. 
\begin{enumerate}[label={(1-\textrm{\roman*})}]
\item there exist facets $F, G\in\mathcal{F}(\Delta_{\bf a})$ such that $F\supset{\rm supp}({\bf a})\cap F_{0}$, $G\supset{\rm supp}({\bf a})\cap G_{0}$ and $F\cap G\neq\emptyset$, 
\item there exists a facet $H\in\mathcal{F}(\Delta_{\bf a})$ such that $H\cap F_{0}\neq\emptyset$ and $H\cap G_{0}\neq\emptyset$.  
\end{enumerate}
Notice that for any facet $F\in\mathcal{F}(\Delta_{\bf a})$, we have one of the following conditions:
\begin{enumerate}[label=(\alph*)]
\item $F\supset{\rm supp}({\bf a})\cap F_{0}$,  
\item $F\supset{\rm supp}({\bf a})\cap G_{0}$, 
\item $F\cap F_{0}\neq\emptyset$, $F\cap G_{0}\neq\emptyset$ and $\displaystyle\sum_{i\in F}a_{i}\geq\ell-1$. 
\end{enumerate}
If at least one facet satisfies condition (c), then condition (ii) holds. Hence, we assume that no facet satisfies condition (c). Let $A=\{F\in\mathcal{F}(\Delta_{\bf a})\,\,: F\supset{\rm supp}({\bf a})\cap F_{0}\}$ and $B=\{G\in\mathcal{F}(\Delta_{\bf a})\,\,: G\supset{\rm supp}({\bf a})\cap G_{0}\}$. It is obvious that $F_{0}\in A$ and $G_{0}\in B$. If $F\cap G = \emptyset$ for every $F\in A$ and every $G\in B$, then this contradicts the assumption that $\Delta_{\mathbf{a}}$ is connected. Hence, there exist facets $F\in A$ and $G\in B$ such that $F\cap G \neq\emptyset$, that is, condition (i) holds.

\vspace{0.1cm}
{\it Case 2}: $\sum_{i \in [n] \setminus (F_{0} \cup G_{0})} a_{i}\leq\left\lfloor \tfrac{\ell - 1}{2} \right\rfloor$. 

\vspace{0.1cm}In this, case, we have $F\cap F_{0}\neq\emptyset$ or $F\cap G_{0}\neq\emptyset$. Indeed, by the assumption on the choice of $F_0$ and $G_0$, we obtain that$$\displaystyle\sum_{i\in F_{0}}a_{i}=\displaystyle\sum_{i\in G_{0}}a_{i}=(\ell-1)-\left\lfloor \tfrac{\ell - 1}{2} \right\rfloor.$$Hence we obtain that $\displaystyle\sum_{i\in F_{0}\cup G_{0}}a_{i}\geq\ell-1$. Therefore, if $F\cap F_{0}=\emptyset$ and $F\cap G_{0}=\emptyset$, then $\displaystyle\sum_{i\in F}a_{i}\geq\ell$ and  $\displaystyle\sum_{i\in G}a_{i}\geq\ell$, which leads to a contradiction the assumption that $F, G\in\mathcal{F}(\Delta_{\bf a})$. Therefore, since $\Delta_{\bf a}$ is connected, there exist facets $F, G\in\mathcal{F}(\Delta_{\bf a})$ such that $F\cap F_{0}\neq\emptyset$, $G\cap G_{0}\neq\emptyset$ and $F\cap G\neq\emptyset$. 

\vspace{0.1cm}
{\it Case 3}:  $q=\sum_{i \in [n] \setminus (F_{0} \cup G_{0})} a_{i}>\left\lfloor \tfrac{\ell - 1}{2} \right\rfloor$ and $p$ divides $q$, where $p=\displaystyle\sum_{i\in F_{0}}a_{i}=\displaystyle\sum_{i\in G_{0}}a_{i}$. 

\vspace{0.1cm}In this case, we prove that one of the following conditions holds. 
\begin{enumerate}[label={(3-\textrm{\roman*})}]
\item there exist facets $F, G\in\mathcal{F}(\Delta_{\bf a})$ such that $F\supset{\rm supp}({\bf a})\cap F_{0}$, $G\supset{\rm supp}({\bf a})\cap G_{0}$ and $F\cap G\neq\emptyset$, 
\item there exists a facet $H\in\mathcal{F}(\Delta_{\bf a})$ such that $H\cap F_{0}\neq\emptyset$ and $H\cap G_{0}\neq\emptyset$.  
\end{enumerate}
Notice that $p+q=\ell-1$. Let ${\bf a}^{\prime}=(a_{1}^{\prime},\ldots, a_{n}^{\prime})\in\mathbb{N}^{n}$, where 
\[
a_{i}^{\prime} =
\begin{cases}
(\frac{q}{p}+1)a_{i} & (\text{if } i \in F_{0}\cup G_{0}), \\
0 & \text{otherwise}.
\end{cases}
\]
Then, since $\displaystyle\sum_{i\in F_{0}}a_{i}^{\prime}=\displaystyle\sum_{i\in G_{0}}a_{i}^{\prime}=\ell-1$, from the conditions (1-i) and (1-ii) in the case 1, there exist facets $F, G\in\mathcal{F}(\Delta_{{\bf a}^{\prime}})$ such that $F\supset{\rm supp}({\bf a})\cap F_{0}$, $G\supset{\rm supp}({{\bf a}^{\prime}})\cap G_{0}$ and $F\cap G\neq\emptyset$, or there exists a facet $H\in\mathcal{F}(\Delta_{{\bf a}^{\prime}})$ such that $H\cap F_{0}\neq\emptyset$ and $H\cap G_{0}\neq\emptyset$. It remains to verify that $F$, $G$, and $H$ are facets of $\Delta_{\bf a}$. For $F$, since $$\displaystyle\sum_{i\notin F}a_{i}\leq p+q=\displaystyle\sum_{i\notin F}a_{i}^{\prime}\leq\ell-1,$$it follows that $F$ is a facet of $\Delta_{\bf a}$ from Lemma \ref{symbolic power cpx}. The same argument holds for $G$ as well. For $H$, since $H\in\mathcal{F}(\Delta_{{\bf a}^{\prime}})$ and $\displaystyle\sum_{i\in[n]}a_{i}^{\prime}=2\ell-2$, we have $$\ell-1\leq\displaystyle\sum_{i\in H}a_{i}^{\prime}=(\tfrac{q}{p}+1)\displaystyle\sum_{i\in H}a_{i}.$$Hence, from $p+q=\ell-1$, we obtain that $$\displaystyle\sum_{i\in H}a_{i}\geq\frac{\ell-1}{\frac{q}{p}+1}=\frac{p(\ell-1)}{p+q}=\frac{p(\ell-1)}{\ell-1}=p,$$ and thus we have $$\displaystyle\sum_{i\notin H}a_{i}\leq p+q=\ell-1$$ since $\displaystyle\sum_{i\in[n]}a_{i}=2p+q$. Therefore, we get $H\in\mathcal{F}(\Delta_{\bf a})$.   

\vspace{0.1cm}
{\it Case 4}:  $q=\sum_{i \in [n] \setminus (F_{0} \cup G_{0})} a_{i}>\left\lfloor \tfrac{\ell - 1}{2} \right\rfloor$ and $p$ does not divide $q$, where $p=\displaystyle\sum_{i\in F_{0}}a_{i}=\displaystyle\sum_{i\in G_{0}}a_{i}$. 

\vspace{0.1cm}In this case, we prove that there exist facets $F, G\in\Delta_{\bf a}$ such that $F\cap F_{0}\neq\emptyset$, $G\cap G_{0}\neq\emptyset$ and $F\cap G\neq\emptyset$. Observe that there exists a vector ${\bf a}^{\prime}=(a_{1}^{\prime},\ldots, a_{n}^{\prime})\in\mathbb{N}^{n}$ such that if $ i \in F_{0}\cup G_{0}$, then $a_{i}^{\prime}=(\left\lfloor\frac{q}{p}\right\rfloor+1)a_{i}$, otherwise $a_{i}^{\prime}\leq a_{i}$ and $$\sum_{i\in[n]}a_{i}^{\prime}=2p+2\left\lfloor\tfrac{q}{p}\right\rfloor+(q-p\left\lfloor\tfrac{q}{p}\right\rfloor).$$
We set $p^{\prime}=(1+\left\lfloor\tfrac{q}{p}\right\rfloor)p$ and $q^{\prime}=q-p\left\lfloor\tfrac{q}{p}\right\rfloor$. Notice that $p^{\prime}>q^{\prime}$. Then, since $\displaystyle\sum_{i\notin F_{0}}a_{i}^{\prime}=\displaystyle\sum_{i\notin G_{0}}a_{i}^{\prime}=p^{\prime}+q^{\prime}=p+q=\ell-1$ and $\displaystyle\sum_{i\in[n]\setminus(F_{0}\cup G_{0})}a_{i}^{\prime}=q^{\prime}\leq\left\lfloor\tfrac{\ell-1}{2}\right\rfloor$, by the result of Case 2, there exist facets $F, G\in\mathcal{F}(\Delta_{{\bf a}^{\prime}})$ such that $F\cap F_{0}\neq\emptyset$, $G\cap G_{0}\neq\emptyset$ and $F\cap G\neq\emptyset$. It remains to verify that $F$ and $G$ are facets of $\Delta_{\bf a}$. For $F$, we have the following inequality: 
\begin{align*}
p^{\prime}&\leq\displaystyle\sum_{i\in F}a_{i}^{\prime} \\
&=\left(\displaystyle\sum_{i\in(F_{0}\cup G_{0})\cap F}a_{i}^{\prime}\right)+\left(\displaystyle\sum_{i\in([n]\setminus(F_{0}\cap G_{0}))\cap F}a_{i}^{\prime}\right) \\
&\leq \left(\left(\left\lfloor\tfrac{q}{p}\right\rfloor+1\right)\displaystyle\sum_{i\in(F_{0}\cup G_{0})\cap F}a_{i}\right)+\left(\displaystyle\sum_{i\in([n]\setminus(F_{0}\cap G_{0}))\cap F}a_{i}\right) \\
&\leq \left(\left(\left\lfloor\tfrac{q}{p}\right\rfloor+1\right)\displaystyle\sum_{i\in F}a_{i}\right).
\end{align*}
Therefore, we obtain that $$\displaystyle\sum_{i\in F}a_{i}\geq\frac{p^{\prime}}{\left\lfloor\tfrac{q}{p}\right\rfloor+1}=\frac{(\left\lfloor\tfrac{q}{p}\right\rfloor+1)p}{\left\lfloor\tfrac{q}{p}\right\rfloor+1}=p,$$and hence we see that $F$ is a facet of $\Delta_{\bf a}$ from Lemma \ref{symbolic power cpx}. The same argument holds for $G$ as well. This completes the proof that (1) implies (2). To prove the converse, we assume that (2) holds. Fix $F_{0}, G_{0}\in\mathcal{F}(\Delta)$ with $F_{0}\cap G_{0}\neq\emptyset$ and any ${\bf a}=(a_{1},\ldots, a_{n})\in\mathbb{N}^{n}$ such that $\displaystyle\sum_{i\in[n]}a_{i}\geq\ell$, $\displaystyle\sum_{i\notin F_{0}}a_{i}=\ell-1$ and $\displaystyle\sum_{i\notin G_{0}}a_{i}=\ell-1$. It suffices to prove that there exist facets $F, G\in\mathcal{F}(\Delta_{\bf a})$ such that $F\cap F_{0}\neq\emptyset$, $G\cap G_{0}\neq\emptyset$ and $F\cap G\neq\emptyset$. Observe that there exists a vector ${\bf a}^{\prime}=(a_{1}^{\prime},\ldots, a_{n}^{\prime})\in\mathbb{N}^{n}$ such that $a_{i}^{\prime}\geq a_{i}$ for all $i$ and $\displaystyle\sum_{i\notin F_{0}}a_{i}^{\prime}=\displaystyle\sum_{i\notin G_{0}}a_{i}^{\prime}=\ell-1$. Now, by the assumption, with respect to ${\bf a}^{\prime}$, there exist facets $F, G\in\mathcal{F}(\Delta_{{\bf a}^{\prime}})$ such that $F\cap F_{0}\neq\emptyset$, $G\cap G_{0}\neq\emptyset$ and $F\cap G\neq\emptyset$. Since  $\ell-1\geq\displaystyle\sum_{i\notin F}a_{i}^{\prime}\geq\displaystyle\sum_{i\notin F}a_{i}$, we see that $F\in\mathcal{F}(\Delta_{\bf a})$. The same argument holds for $G$ as well. Therefore, $\Delta_{\bf a}$ is connected, as required. 
\end{proof}

We have the following characterization of $S_{2}\mbox{-}{\rm depth}(S/I_{\Delta}^{(\ell)})$ by combining Proposition \ref{S_2-depth alg} and Theorem \ref{1st local coho}.
\begin{cor}\label{combi of S_2-depth}
Let $\Delta$ be a simplicial complex and $\ell\geq1$. Then we have 
$$S_{2}\mbox{-}{\rm depth}(S/I_{\Delta}^{(\ell)})=\min\{|F|\,\,: F\in\Delta\mbox{ satisfies condition }(*)\}+1,$$
where $(*):$ there exist facets $F_{0}, G_{0}\in\mathcal{F}({\rm link}_{\Delta}F)$ with $F_{0}\cap G_{0}=\emptyset$ and ${\bf a}\in\mathbb{N}^{V({\rm link}_{\Delta}F)}$ such that $\displaystyle\sum_{i\in[V({\rm link}_{\Delta}F)]}a_{i}\geq\ell$, $\displaystyle\sum_{i\notin F_{0}}a_{i}=\ell-1$ and $\displaystyle\sum_{i\notin G_{0}}a_{i}=\ell-1$, there are no facets $F, G\in\mathcal{F}(\Delta_{\bf a}(I_{{\rm link}_{\Delta}F}^{(\ell)}))$ such that $F\cap F_{0}\neq\emptyset$, $G\cap G_{0}\neq\emptyset$ and $F\cap G\neq\emptyset$. 
\end{cor}

Moreover, we generalize one of the implications of the equivalence given in \cite[Theorem 3.2]{rty} by Theorem \ref{1st local coho}. To this end, we introduce the following notions. Let $\Delta$ be a connected simplicial complex with $\dim\Delta=d-1$ and let $1\leq t\leq d-1$. We denote by $\mathcal{F}^{i}(\Delta)$ the set of faces of $\Delta$ of dimension $i$. Let $F, G\in\mathcal{F}^{t}(\Delta)$. Then, we define the $t$-{\it distance} of $\Delta$, denoted by $t\mbox{-}{\rm dist}(F, G)$ as the minimal length of a sequence 
$F\subset F_{1}, \ldots, F_{p}\supset G$ in $\mathcal{F}^{t+1}(\Delta)$ such that $F_{i}\cap F_{i+1} \neq \emptyset$ for all $i=1, \ldots, p-1$. Also we define the $t$-{\it diameter}, denoted by $t\mbox{-}{\rm diam}(\Delta)$, is the maximal $t$-distance between faces in $\mathcal{F}^{t+1}(\Delta)$.  

\begin{cor}\label{gene of rty}
Let $\Delta$ be a simplicial complex with $\dim\Delta=d-1$ and let $2\leq \ell\leq d-1$. If $(\ell-1)\mbox{-}{\rm diam}(\Delta)\leq2$, then we have $H_{\mathfrak{m}}^{1}(S/I_{\Delta}^{(\ell)})=0$ and hence ${\rm depth}(S/I_{\Delta}^{(\ell)})\geq2$.  
\end{cor}
\begin{proof}
Fix $F_{0},G_{0}\in\mathcal{F}(\Delta_{\bf a}(I_{\Delta}^{(\ell)}))$ with $F_{0}\cap G_{0}=\emptyset$ and ${\bf a}\in\mathbb{N}^{n}$ such that  $\sum_{i\in[n]}a_{i}\geq\ell$, $\sum_{i\notin F_{0}}a_{i}=\ell-1$ and $\sum_{i\notin G_{0}}a_{i}=\ell-1$. Since $\sum_{i\in F_{0}}a_{i}\leq\sum_{i\notin G_{0}}a_{i}\leq\ell-1$, there exists a face $F_{1}$ of $\Delta_{\bf a}(I_{\Delta}^{(\ell)})$ such that $\dim F_{1}=\ell-1$ and $\sum_{i\in F_{0}\setminus F_{1}}a_{i}=0$. Similarly, we see that there exists a face $G_{1}$ of $\Delta_{\bf a}(I_{\Delta}^{(\ell)})$ such that $\dim G_{1}=\ell-1$ and $\sum_{i\in G_{0}\setminus G_{1}}a_{i}=0$. Since $(\ell-1)\mbox{-}{\rm diam}(\Delta)\leq2$, there exist facets $F$ and $G$ of $\Delta$ such that $F_{1}\subset F$, $G_{0}\subset G$ and $F\cap G\neq\emptyset$. We must check that $F, G\in\mathcal{F}(\Delta_{\bf a}(I_{\Delta}^{(\ell)}))$. From $\sum_{i\in F_{0}\setminus F_{1}}a_{i}=0$ and $\sum_{i\notin F_{0}}a_{i}=\ell-1$, we obtain that $$\sum_{i\notin F}a_{i}\leq\left(\sum_{i\in F_{0}\setminus F_{1}}a_{i}\right)+\left( \sum_{i\notin F_{0}\cup G_{0}}a_{i}\right)+\left(\sum_{i\in G_{0}}a_{i}\right)=\ell-1.$$Similarly, we see that $\sum_{i\notin G}a_{i}\leq\ell-1$. Hence, by Lemma \ref{symbolic power cpx}, we have $F, G\in\mathcal{F}(\Delta_{\bf a}(I_{\Delta}^{(\ell)}))$. Therefore, by Theorem \ref{1st local coho}, the assertion follows. 
\end{proof}

By Proposition \ref{S_2-depth alg} and Corollary \ref{gene of rty}, we obtain a combinatorial lower bound for $S_{2}\mbox{-}{\rm depth}(S/I_{\Delta}^{(\ell)})$ in terms of $\Delta$. 
\begin{cor}
For a simplicial complex $\Delta$ and $\ell\geq2$, we have 
$$S_{2}\mbox{-}{\rm depth}(S/I_{\Delta}^{(\ell)})\geq\min\{|F|\,\,:F\in\Delta\mbox{ with }(\ell-1)\mbox{-}{\rm diam}({\rm link}_{\Delta}F)>2\}+1.$$
\end{cor}

Although we expected the Serre depth for $(S_{2})$ on the symbolic powers of an arbitrary Stanley–Reisner ideal to be non-increasing, we construct a simplicial complex whose the Serre depth for $(S_{2})$ and the depth fail to satisfy this property. First, let us recall the definition of shellable simplicial complexes. Let $\Delta$ be a simplicial complex on $[n]$. Then $\Delta$ is called {\it shellable}, if its facets can be ordered $F_{1},\ldots, F_{m}$ such that, for all $2\leq m$, the subcomplex $\langle F_{1},\ldots, F_{j-1}\rangle\cap\langle F_{j}\rangle$ is pure of dimension $\dim F_{j}-1$. It is known that this condition is equivalent to that there exists a facet ordering $F_{1},\ldots, F_{m}$ such that for all $i$ and all $j<i$, there exists $t\in F_{i}\setminus F_{j}$ and $k<i$ such that $F_{i}\setminus F_{k}=\{t\}$ (see, e.g. \cite{hh}). Moreover, it is known that if $\Delta$ is pure and shellable, then its Stanley--Reisner ring is Cohen--Macaulay over an arbitrary field (see, e.g. \cite[Theorem 8.2.6]{hh}). 

\begin{thm}\label{CE of non-increasing}
Suppose that $d\geq 3$. Let $\Delta$ be the simplicial complex on the vertex set $\{x_{1},\ldots, x_{d}, y_{1},\ldots, y_{d}, z\}$ whose facets are
\begin{enumerate}[label=(\alph*)]
\item $\{x_{1},\ldots x_{d}\}$, 
\item $\{y_{1},\ldots, y_{d}\}$, 
\item $\{\{x_{1},\ldots, x_{d}, z\}\setminus\{x_{i}\}\,\,: 1\leq i\leq d\}$, 
\item $\{\{y_{1},\ldots, y_{d}, z\}\setminus\{y_{j}\}\,\,: 1\leq j\leq d\}$, 
\item $\left\{\{x_{i_{1}},\ldots, x_{i_{k}}, y_{j_{1}},\ldots, y_{j_{d-k-1}}, z\}:\, \begin{aligned}[t] &1\leq k\leq d-1,\; 1\leq i_{1}<\cdots<i_{k}\leq d,\\ &\mbox{ and }1\leq j_{1}<\cdots<j_{d-k-1}\leq d \end{aligned} \right\}$. 
\end{enumerate}
Then, we have 
\begin{enumerate}
\item $\Delta$ is pure and shellable, 
\item $S/I_{\Delta}^{(2)}$ is Cohen--Macaulay, 
\item ${\rm depth}(S/I_{\Delta}^{(d+1)})=1\mbox{ and }{\rm depth}(S/I_{\Delta}^{(d+2)})\geq2$, 
\item $S_{2}\mbox{-}{\rm depth}(S/I_{\Delta}^{(d+1)})=1\mbox{ and }S_{2}\mbox{-}{\rm depth}(S/I_{\Delta}^{(d+2)})\geq2$. 
\end{enumerate}
\end{thm}
\begin{proof}
$(1):$ We consider the ordering of facets according to the number of elements among $x_{1}, \ldots, x_{d}$ contained in each facet, with facets containing more of these elements placed first. When two facets contain the same number of such elements, they are ordered in decreasing lexicographic order with respect to $x_{1}>\cdots>x_{d}>z>y_{1}>\cdots>y_{d}$. We denote by $F_{1},\ldots,F_{m}$ the sequence of facets according to this order. Fix $j<i$ and we prove that there exists $u\in F_{i}\setminus F_{j}$ and $k<i$ such that $F_{i}\setminus F_{k}=\{u\}$. To this end, we distinguish the following cases:

\vspace{0.1cm}
{\it Case} I: There exists $p$ such that $F_{i}=\{x_{1},\ldots, x_{d},z\}\setminus\{x_{p}\}$. 

\vspace{0.1cm}The ordering of the facets leads to the following cases:
\begin{enumerate}[label={(I-\textnormal{\roman*})}]
\item $F_{j}=\{x_{1},\ldots, x_{d}\}$, 
\item There exists $q$ with $p<q$ such that $F_{j}=\{x_{1},\ldots, x_{d},z\}\setminus\{x_{q}\}$. 
\end{enumerate}
If condition (I-i) holds, then $F_{i}\setminus F_{j}=\{z\}$, while if condition (I-ii) holds, then $F_{i}\setminus F_{j}=\{x_{q}\}$. In either case, the assertion follows.

\vspace{0.1cm}
{\it Case} II: There exist $x_{s_{1}},\ldots, x_{s_{t}}$ and $y_{r_{1}},\ldots, y_{r_{d-t-1}}$ with $s_{1}<\cdots<s_{t}$ and $r_{1}<\cdots<r_{d-t-1}$ such that $F_{i}=\{x_{s_{1}},\ldots, x_{s_{t}}, y_{r_{1}},\ldots, y_{r_{d-t-1}}, z\}$. 

\vspace{0.1cm}The ordering of the facets leads to the following distinctions:
\begin{enumerate}[label={(II-\textrm{\roman*})}]
\item $F_{j}=\{x_{1},\ldots, x_{d}\}$, 
\item There exists $q$ such that $F_{j}=\{x_{1},\ldots, x_{d},z\}\setminus\{x_{q}\}$. 
\item There exist $x_{s_{1}^{\prime}},\ldots, x_{s_{t^{\prime}}^{\prime}}$ and $y_{r_{1}^{\prime}},\ldots, y_{r_{d-t^{\prime}-1}^{\prime}}$ with $s_{1}^{\prime}<\cdots<s_{t^{\prime}}^{\prime}$ and $r_{1}^{\prime}<\cdots<r_{d-t^{\prime}-1}^{\prime}$ such that $F_{j}=\{x_{s_{1}^{\prime}},\ldots, x_{s_{t^{\prime}}^{\prime}}, y_{r_{1}^{\prime}},\ldots, y_{r_{d-t^{\prime}-1}^{\prime}}, z\}$. 
\end{enumerate}
If either condition (II-i) or (II-ii) is satisfied, then we have $F_{i}\setminus F_{j}\supset\{y_{r_{1}},\ldots, y_{r_{d-k-1}}, z\}$.  By considering a facet $F_{k}$ defined as $\{x_{s_{1}},\ldots, x_{s_{1}}, x_{s_{t+1}}, y_{r_{1}},\ldots, y_{r_{d-k-2}}, z\}$, we have $F_{i}\setminus F_{k}=\{y_{r_{d-k-1}}\}$ and by the definition of a facet ordering, we see that $k<i$. We suppose that the condition (II-iii) is satisfied. If $y_{r_{u}}\in F_{i}\setminus F_{j}$, then we may take $F_{k}$ to be $\{x_{s_{1}},\ldots, x_{s_{t}}, x_{s_{t+1}}, y_{r_{1}},\ldots, y_{r_{d-t-1}}, z\}\setminus\{y_{r_{u}}\}$, where $x_{s_{t+1}}\notin\{x_{s_{1}},\ldots, x_{s_{t}}\}$. Then we have $F_{i}\setminus F_{k}=\{y_{r_{u}}\}$ and by the definition of the ordering of facets, we see that $k<i$. Hence we may assume that $F_{i}\setminus F_{j}\subset\{x_{1},\ldots, x_{d}\}$. Suppose that $x_{s_{v}}\in F_{i}\setminus F_{j}$. If $x_{s_{v}}=x_{1}$, then since $\{y_{r_{1}},\ldots, y_{r_{d-t-1}}\}=\{y_{r_{1}^{\prime}},\ldots, y_{r_{d-t^{\prime}-1}^{\prime}}\}$ and the facets are ordered based on how many of $x_{1},\ldots, x_{s}$ they contain, this case does not arise. Moreover, if $s_{1},\ldots, s_{t}$ are consecutive, then $x_{s_{1}}\neq x_{1}$ since the facets are ordered by lexicographic order induced by $x_{1}>\cdots>x_{d}>z>y_{1}>\cdots>y_{d}$. Therefore it is enough to consider the situation where these cases do not happen. In this case, we may take $F_{k}$ to be $(\{x_{s_{1}},\ldots, x_{s_{t}}, y_{r_{1}},\ldots, y_{r_{d-t-1}}, z\}\setminus\{x_{s_{v}}\})\cup\{x_{s_{w}}\}$, where $s_{w}<s_{v}$ and $x_{s_{w}}\notin\{x_{s_{1}},\ldots, x_{s_{t}}\}$.  

\vspace{0.1cm}
{\it Case} III: There exists $p$ such that $F_{i}=\{y_{1},\ldots, y_{d},z\}\setminus\{y_{p}\}$. 

\vspace{0.1cm}The ordering of the facets leads to the following distinctions:
\begin{enumerate}[label={(III-\textrm{\roman*})}]
\item $F_{j}=\{x_{1},\ldots, x_{d}\}$, 
\item There exists $q$ such that $F_{j}=\{x_{1},\ldots, x_{d},z\}\setminus\{x_{q}\}$. 
\item There exist $x_{s_{1}^{\prime}},\ldots, x_{s_{t^{\prime}}^{\prime}}$ and $y_{r_{1}^{\prime}},\ldots, y_{r_{d-t^{\prime}-1}^{\prime}}$ with $s_{1}^{\prime}<\cdots<s_{t^{\prime}}^{\prime}$ and $r_{1}^{\prime}<\cdots<r_{d-t^{\prime}-1}^{\prime}$ such that $F_{j}=\{x_{s_{1}^{\prime}},\ldots, x_{s_{t^{\prime}}^{\prime}}, y_{r_{1}^{\prime}},\ldots, y_{r_{d-t^{\prime}-1}^{\prime}}, z\}$. 
\end{enumerate}
If either condition (III-i) or (III-ii) is satisfied, then we have $F_{i}\setminus F_{j}=\{y_{1},\ldots, y_{d}, z\}\setminus\{y_{p}\}$.  Take $q$ be an integer with $p\neq q$. By considering a facet $F_{k}$ defined as $\{x_{1}, y_{1},\ldots, y_{d}, z\}\setminus\{y_{p}, y_{q}\}$, we have $F_{i}\setminus F_{k}=\{y_{q}\}$ and by the definition of a facet ordering, we see that $k<i$. We suppose that the condition (III-iii) is satisfied. In this case, we have $F_{i}\setminus F_{j}\subset\{y_{1},\ldots, y_{d}\}$. Let $y_{s_{v}}\in F_{i}\setminus F_{j}$. Then, by considering a facet $F_{k}$ as $\{x_{1}, y_{1},\ldots, y_{d}, z\}\setminus\{y_{p}, y_{s_{v}}\}$, we obtain that $F_{i}\setminus F_{k}=\{y_{s_{v}}\}$ and $k<i$. 

\vspace{0.1cm}
{\it Case} IV: $F_{i}=\{y_{1},\ldots, y_{d}\}$. 

\vspace{0.1cm}The ordering of the facets leads to the following distinctions:
\begin{enumerate}[label={(IV-\textrm{\roman*})}]
\item $F_{j}=\{x_{1},\ldots, x_{d}\}$, 
\item There exists $q$ such that $F_{j}=\{x_{1},\ldots, x_{d},z\}\setminus\{x_{q}\}$. 
\item There exist $x_{s_{1}^{\prime}},\ldots, x_{s_{t^{\prime}}^{\prime}}$ and $y_{r_{1}^{\prime}},\ldots, y_{r_{d-t^{\prime}-1}^{\prime}}$ with $s_{1}^{\prime}<\cdots<s_{t^{\prime}}^{\prime}$ and $r_{1}^{\prime}<\cdots<r_{d-t^{\prime}-1}^{\prime}$ such that $F_{j}=\{x_{s_{1}^{\prime}},\ldots, x_{s_{t^{\prime}}^{\prime}}, y_{r_{1}^{\prime}},\ldots, y_{r_{d-t^{\prime}-1}^{\prime}}, z\}$. 
\item  There exists $p$ such that $F_{j}=\{y_{1},\ldots, y_{d},z\}\setminus\{y_{p}\}$. 
\end{enumerate}
For all cases, we may consider $y_{m}\in F_{i}\setminus F_{j}$. Then, by considering a facet $F_{k}$ as $\{y_{1},\ldots, y_{d}, z\}\setminus\{y_{m}\}$, we obtain that $F_{i}\setminus F_{k}=\{y_{m}\}$ and $k<i$. Considering all the cases above, the proof is complete. 

$(2):$ From \cite[Theorem 2.1]{mt1}, it is known that $S/I_{\Delta}^{(\ell)}$ is Cohen--Macaulay if and only if $\Delta$ is Cohen--Macaulay and $\Delta_{V}$ is Cohen--Macaulay for a subset $V$ of vertices of $\Delta$ with $2\leq|V|\leq d$, where $\Delta_{V}$ is the subcomplex of $\Delta$ whose facets are the facets of $\Delta$ with at least $|V|-1$ vertices in $V$. Hence, by (1), it is enough to prove that $\Delta_{V}$ is Cohen--Macaulay for any subset $V$ of vertices of $\Delta$ with $2\leq|V|\leq d$. We distinguish the following cases:

\vspace{0.1cm}
{\it Case} I: $V\subset\{x_{1},\ldots, x_{d}, z\}$. 

\vspace{0.1cm}In this case, we consider the ordering of facets induced by (1). Fix $j<i$. We now prove that there exists $u\in F_{i}\setminus F_{j}$ and $k<i$ such that $F_{i}\setminus F_{k}=\{u\}$ in $\Delta_{V}$.
If $F_{i}$ is of the form $\{x_{1},\ldots, x_{d}, z\}\setminus\{x_{p}\}$ for some $p$, then $F_{i}\setminus F_{1}=\{z\}$ and $F_{i}\setminus F_{j}=\{x_{q}\}$ for any $2\leq j<i$ and some $q$.  
Suppose that there exist $x_{s_{1}},\ldots, x_{s_{t}}$ and $y_{r_{1}},\ldots, y_{r_{d-t-1}}$ with $s_{1}<\cdots<s_{t}$ and $r_{1}<\cdots<r_{d-t-1}$ such that $F_{i}=\{x_{s_{1}},\ldots, x_{s_{t}}, y_{r_{1}},\ldots, y_{r_{d-t-1}}, z\}$. Suppose that there exists $y_{r_{v}}\in F_{i}\setminus F_{j}$. Then, we may take $F_{k}$ to be $\{x_{s_{1}},\ldots, x_{s_{t}}, x_{s_{t+1}}, y_{r_{1}},\ldots, y_{r_{d-t-1}}, z\}\setminus\{y_{r_{v}}\}$, where $x_{s_{t+1}}\notin\{x_{s_{1}},\ldots, x_{s_{t}}\}$. Since $|F_{k}\cap V|\geq|F_{i}\cap V|\geq|V|-1$, $F_{k}$ is a facet of $\Delta_{V}$. Suppose that $F_{i}\setminus F_{j}\subset\{x_{1},\ldots, x_{d}\}$. Fix $x_{s_{v}}\in F_{i}\setminus F_{j}$. Then we may take $F_{k}$ to be $\{x_{s_{1}},\ldots, x_{s_{t}}, x_{s_{t+1}}, y_{r_{1}},\ldots, y_{r_{d-t-1}}, z\}\setminus\{x_{s_{v}}\}$, where $x_{s_{t+1}}\notin\{x_{s_{1}},\ldots, x_{s_{t}}\}$. If $x_{s_{v}}\notin V$, then $|F_{k}\cap V|=|F_{i}\cap V|\geq |V|-1$ and hence $F_{k}$ is a facet of $\Delta_{V}$. Suppose that $x_{s_{v}}\in V$. Since $|F_{j}\cap V|=|V|-1$, there exists $x_{s_{w}}\in F_{j}\setminus F_{i}$ such that $x_{s_{w}}\in V$, and hence by setting $x_{s+1}$ as $x_{s_{w}}$, we obtain that $F_{k}\in\mathcal{F}(\Delta_{V})$. Therefore, $\Delta_{V}$ is pure and shellable and thus, Cohen--Macaulay.  

\vspace{0.1cm}
{\it Case} II: $V\subset\{y_{1},\ldots, y_{d}, z\}$. 

\vspace{0.1cm}In this case, by symmetry and Case 1, this case follows by considering the ordering of facets with respect to $y_{1}, \ldots, y_{d}$ instead of $x_{1}, \ldots, x_{d}$. 

\vspace{0.1cm}
{\it Case} III: $V\subset\{x_{1},\ldots, x_{d}\}\cup\{y_{1},\ldots, y_{d}\}$ and $V$ does not satisfy Case I and Case II. 

\vspace{0.1cm}Fix $j<i$ and  a subset $V$ of vertices of $\Delta$ with $2\leq|V|\leq d$. First we suppose that $|V\cap\{x_{1},\ldots, x_{d}\}|\geq2$ and $|V\cap\{y_{1},\ldots, y_{d}\}|\geq2$. Then, by the definition of $\Delta_{V}$ and $z\notin V$, we have $\mathcal{F}(\Delta_{V})=\{F\in\mathcal{F}(\Delta)\,\,: V\subset F\}.$ Notice that an elements of $F_{i}\setminus F_{j}$ is not contained in $V$ since $F_{i}$ and $F_{j}$ contain $V$. Hence, if there exists $y_{r_{v}}$ such that $y_{r_{v}}\in F_{i}\setminus F_{j}$, then by considering a facet $F_{k}$ as $\{x_{s_{1}},\ldots, x_{s_{t}}, x_{s_{t+1}}, y_{r_{1}},\ldots, y_{r_{d-t-1}}, z\}\setminus\{y_{r_{v}}\}$, where $x_{s_{t+1}}\notin\{x_{s_{1}},\ldots, x_{s_{t}}\}$, $F_{k}$ is a facet of $\Delta_{V}$, $F_{i}\setminus F_{k}=\{y_{r_{v}}\}$ and $k<i$. Hence we suppose that $F_{i}\setminus F_{j}\subset\{x_{1},\ldots, x_{d}\}$. We suppose that there exists $x_{s_{v}}$ such that $x_{s_{v}}\in F_{i}\setminus F_{j}$. Then, as in case (II-iii) of (1), there exists $x_{s_{w}}$ such that $s_{w}<s_{v}$ and $x_{s_{w}}\notin\{x_{s_{1}},\ldots, x_{s_{t}}\}$, $F_{i}\setminus F_{k}=\{x_{s_{v}}\}$. Hence, by considering $F_{k}$ as $(\{x_{s_{1}},\ldots, x_{s_{t}}, y_{r_{1}},\ldots, y_{r_{d-t-1}}, z\}\setminus\{x_{s_{v}}\})\cup\{x_{s_{w}}\}$, we obtain that $F_{k}$ is a facet of $\Delta_{V}$, $F_{i}\setminus F_{k}=\{x_{s_{v}}\}$ and $k<i$, which implies that $\Delta_{V}$ is pure and shellable, and hence Cohen--Macaulay. Next, we suppose that $|V\cap\{x_{1},\ldots, x_{d}\}|=1$ or $|V\cap\{y_{1},\ldots, y_{d}\}|=1$. In the case that of $|V\cap\{x_{1},\ldots, x_{d}\}|=1$, we have $\mathcal{F}(\Delta_{V})=\{F\in\mathcal{F}(\Delta)\,\,: V\subset F\}\cup\{x_{1},\ldots, x_{d}\}$. Since $F_{i}\setminus F_{1}\subset\{x_{1},\ldots, x_{d}\}$, as in the previous discussion, the assertion follows. In the case that of $|V\cap\{y_{1},\ldots, y_{d}\}|=1$, by symmetry, the result follows by considering the same ordering of facets as in Case II. Therefore, for all cases, $\Delta_{V}$ is pure and shellable, and hence Cohen--Macaulay. 

\vspace{0.1cm}
{\it Case} IV: $z\in V$ and $V$ does not satisfy Case I and Case II. 

\vspace{0.1cm}Fix $j<i$ and  a subset $V$ of vertices of $\Delta$ with $2\leq|V|\leq d$. First, we suppose that $|V\cap\{y_{1},\ldots, y_{d}\}|=1$. If $F_{i}$ is of the form $\{x_{1},\ldots, x_{d}, z\}\setminus\{x_{p}\}$ for some $p$, then $F_{i}\setminus F_{1}=\{z\}$ and $F_{i}\setminus F_{j}=\{x_{q}\}$, where $F_{j}=\{x_{1},\ldots, x_{d}, z\}\setminus\{x_{q}\}$. Hence we suppose that there exist $x_{s_{1}},\ldots, x_{s_{t}}$ and $y_{r_{1}},\ldots, y_{r_{d-t-1}}$ with $s_{1}<\cdots<s_{t}$ and $r_{1}<\cdots<r_{d-t-1}$ such that $F_{i}=\{x_{s_{1}},\ldots, x_{s_{t}}, y_{r_{1}},\ldots, y_{r_{d-t-1}}, z\}$. Then $F_{i}\setminus F_{1}=\{y_{r_{1}},\ldots, y_{r_{d-t-1}}, z\}$. Fix $y_{r_{v}}\in F_{i}\setminus F_{1}$. We consider a facet $F_{k}$ defined as $\{x_{s_{1}},\ldots, x_{s_{t}}, x_{s_{t+1}} y_{r_{1}},\ldots, y_{r_{d-t-1}}, z\}\setminus\{y_{r_{v}}\}$, where $x_{s_{t+1}}\notin\{x_{s_{1}},\ldots, x_{s_{t}}\}$. Then $F_{i}\setminus F_{k}=\{y_{r_{v}}\}$. It is enough to check that $F_{k}$ is a facet of $\Delta_{V}$. If $y_{r_{v}}\notin V$, then we have $|F_{k}\cap V|=|F_{i}\cap V|\geq|V|-1$, and hence $F_{k}\in\mathcal{F}(\Delta_{V})$. Hence we may assume that $y_{r_{v}}\in V$. If $|\{x_{s_{1}},\ldots, x_{s_{t}}\}\cap V|=|V|-1$, then it is clear that $F_{k}\in\mathcal{F}(\Delta_{V})$. If $|\{x_{s_{1}},\ldots, x_{s_{t}}\}\cap V|=|V|-2$, then, by the assumption that $|V\cap\{y_{1},\ldots, y_{d}\}|=1$, there exists $x_{s_{w}}\in V$ such that $x_{s_{w}}\notin\{x_{s_{1}},\ldots, x_{s_{t}}\}$. Therefore, by setting $x_{s_{t+1}}$ as $x_{s_{w}}$, we obtain that $F_{k}\in\mathcal{F}(\Delta_{V})$. If there exists $q$ such that $F_{j}=\{x_{1},\ldots, x_{d}, z\}\setminus\{x_{q}\}$, then there exists $y_{r_{v}}$ such that $y_{r_{v}}\in F_{i}\setminus F_{j}$. Hence, in this case, as in the previous discussion, there exists a facet $F_{k}$ of $\Delta_{V}$ such that $F_{i}\setminus F_{k}=\{y_{r_{v}}\}$ and $k<i$. Suppose that there exist $x_{s_{1}^{\prime}},\ldots, x_{s_{t^{\prime}}^{\prime}}$ and $y_{r_{1}^{\prime}},\ldots, y_{r_{d-t^{\prime}-1}^{\prime}}$ with $s_{1}^{\prime}<\cdots<s_{t^{\prime}}^{\prime}$ and $r_{1}^{\prime}<\cdots<r_{d-t^{\prime}-1}^{\prime}$ such that $F_{j}=\{x_{s_{1}^{\prime}},\ldots, x_{s_{t^{\prime}}^{\prime}}, y_{r_{1}^{\prime}},\ldots, y_{r_{d-t^{\prime}-1}^{\prime}}, z\}$. We suppose that there exists $x_{s_{v}}$ such that $x_{s_{v}}\in F_{i}\setminus F_{j}$. Then we may take $F_{k}$ to be $\{x_{s_{1}},\ldots, x_{s_{t}}, x_{s_{t+1}}, y_{r_{1}},\ldots, y_{r_{d-t-1}}, z\}\setminus\{x_{s_{v}}\}$, where $x_{s_{t+1}}\notin\{x_{s_{1}},\ldots, x_{s_{t}}\}$. If $x_{s_{v}}\notin V$, then since $|F_{k}\cap V|=|F_{i}\cap V|\geq|V|-1$, $F_{k}$ is a facet of $\Delta_{V}$. Also, if $x_{k}\in V$, then $|F_{j}\cap V|=|V|-1$ and hence $F_{k}$ is a facet of $\Delta_{V}$. Therefore, $F_{k}$ is a facet of $\Delta_{V}$ such that $k<i$ and $F_{i}\setminus F_{k}=\{x_{s_{v}}\}$. Moreover, if there exists $y_{r_{v}}$ such that $y_{r_{v}}\in F_{i}\setminus F_{j}$, then by considering $F_{k}$ as $\{x_{s_{1}},\ldots, x_{s_{t}}, x_{s_{t+1}} y_{r_{1}},\ldots, y_{r_{d-t-1}}, z\}\setminus\{y_{r_{v}}\}$, where $x_{s_{t+1}}\notin\{x_{s_{1}},\ldots, x_{s_{t}}\}$. Then $F_{i}\setminus F_{k}=\{y_{r_{v}}\}$. It is enough to check that $F_{k}$ is a facet of $\Delta_{V}$. If $y_{r_{v}}\notin V$, then we have $|F_{k}\cap V|=|F_{i}\cap V|\geq|V|-1$, and hence $F_{k}\in\mathcal{F}(\Delta_{V})$. Hence we may assume that $y_{r_{v}}\in V$. If $|\{x_{s_{1}},\ldots, x_{s_{t}}\}\cap V|=|V|-1$, then it is clear that $F_{k}\in\mathcal{F}(\Delta_{V})$. If $|\{x_{s_{1}},\ldots, x_{s_{t}}\}\cap V|=|V|-2$, then, by the assumption that $|V\cap\{y_{1},\ldots, y_{d}\}|=1$ and $z\in V$, we see that $|F_{k}\cap V|=|F_{i}\setminus\{y_{r_{v}}\}|=|V|-1$. Hence $F_{k}$ is a facet of $\Delta_{V}$. Therefore, $\Delta_{V}$ is pure and shellable, and hence Cohen--Macaulay. Moreover, if $|V\cap\{x_{1},\ldots, x_{d}\}|=1$, then, by symmetry, the result follows by considering the same ordering of facets as in Case II. Finally, we suppose that $|V\cap\{x_{1},\ldots, x_{d}\}|\geq 2$ and $|V\cap\{y_{1},\ldots, y_{d}\}|\geq2$. In this case, since $\{x_{1},\ldots, x_{d}\}$ and $\{y_{1},\ldots, y_{d}\}$ are not facets of $\Delta_{V}$, $\Delta_{V}$ is a cone with apex at the vertex $z$. Notice that ${\rm link}_{\Delta_{V}}F$ is also a cone with apex at the vertex $z$ for any face $F$ of $\Delta_{V}$ with $z\notin F$. Therefore, by Reisner's criterion (\cite[Theorem 1]{r}), it suffices to prove that ${\rm link}_{\Delta_{V}}z$ is Cohen--Macaulay. We may assume that $\{x_{1},\ldots, x_{d}\}\cap V=\{x_{1},\ldots, x_{p}\}$ and $V\cap\{y_{1},\ldots, y_{d}\}=\{y_{1},\ldots, y_{q}\}$, where $p\geq2$ and $q\geq2$. Notice that the $(d-2)$-skeleton $\langle x_{1},\ldots, x_{d}, y_{1},\ldots, y_{d}\rangle^{d-2}$ is a matroid complex. By \cite[Theorem 3.5]{mt1} or \cite[Theorem 2.1]{v0}, $\Bbbk[x_{1},\ldots, x_{d}, y_{1},\ldots, y_{d}]/I_{\langle x_{1},\ldots, x_{d}, y_{1},\ldots, y_{d}\rangle^{d-2}}^{(2)}$ is Cohen--Macaulay and hence by \cite[Theorem 2.1]{mt1}, $(\langle x_{1},\ldots, x_{d}, y_{1},\ldots, y_{d}\rangle^{d-2})_{V^{\prime}}$ is Cohen--Macaulay for any $V^{\prime}\subset\{x_{1}, \ldots, x_{d}, y_{1},\ldots, y_{d}\}$ with $2\leq|V^{\prime}|\leq d-1$. Therefore, it suffices to prove that $${\rm link}_{\Delta_{V}}z=(\langle x_{1},\ldots, x_{d}, y_{1},\ldots, y_{d}\rangle^{d-2})_{\{x_{1},\ldots, x_{p}, y_{1}, \ldots, y_{q}\}}.$$Fix a facet $F$ of ${\rm link}_{\Delta_{V}}z$, then by the definition $\Delta_{V}$ and the assumption that $z\in V$, we have 
\begin{align*}
|F\cap(\{x_{1},\ldots, x_{p}\}\cup\{y_{1},\ldots, y_{q}\})|&=|F\cap V| \\
&\geq|V|-2 \\
&=|\{x_{1},\ldots, x_{p}\}\cup\{y_{1},\ldots, y_{q}\}|-1.
\end{align*}
Therefore, $F$ is a facet of $(\langle x_{1},\ldots, x_{d}, y_{1},\ldots, y_{d}\rangle^{d-2})_{\{x_{1},\ldots, x_{p}, y_{1}, \ldots, y_{q}\}}.$Conversely, take $\{x_{s_{1}},\ldots, x_{s_{t}}, y_{r_{1}},\ldots, y_{r_{d-t-1}}\}\in(\langle x_{1},\ldots, x_{d}, y_{1},\ldots, y_{d}\rangle^{d-2})_{\{x_{1},\ldots, x_{p}, y_{1}, \ldots, y_{q}\}}$. Then, we see that 
$$|\{x_{s_{1}},\ldots, x_{s_{t}}, y_{r_{1}},\ldots, y_{r_{d-t-1}}\}\cap\{x_{1},\ldots, x_{d}, y_{1},\ldots, y_{d}\}|\geq p+q-1=|V|-2.$$Therefore, $\{x_{s_{1}},\ldots, x_{s_{t}}, y_{r_{1}},\ldots, y_{r_{d-t-1}}, z\}$ is a facet of $\Delta_{V}$, as required.  

$(3):$ Let ${\bf a}=(a_{x_{1}},\ldots, a_{x_{d}},a_{y_{1}},\ldots, a_{y_{d}},a_{z})\in\mathbb{N}^{2d+1}$, where $a_{x_{i}}$ (resp., $a_{y_{j}}$, $a_{z}$) is a non-negative integer corresponding to a vertex $x_{i}$ (resp., $y_{j}$,  $z$). We consider $a_{x_{1}}=\cdots=a_{x_{d}}=a_{y_{1}}=\cdots=a_{y_{d}}=1$ and $a_{z}=0$, namely ${\bf a}=(1,\ldots, 1, 0)$. Then, since $\Delta_{\bf a}(I_{\Delta}^{(d+1)})=\langle F\in\mathcal{F}(\Delta)\,\,: \sum_{i\notin F}a_{i}\leq d\rangle$, we see that $\mathcal{F}(\Delta_{\bf a}(I_{\Delta}^{(d+1)}))=\{\{x_{1},\ldots, x_{d}\}, \{y_{1},\ldots, y_{d}\}\}$, and this complex is disconnected. Therefore, by Theorem \ref{Takayama}, we have $H_{\mathfrak{m}}^{1}(S/I_{\Delta}^{(d+1)})\neq 0$, that is, ${\rm depth}(S/I_{\Delta}^{(d+1)})=1$. We now prove ${\rm depth}(S/I_{\Delta}^{(d+2)})\geq2$. To this end, it is enough to show that $\Delta_{\bf a}(I_{\Delta}^{(d+2)})$ is connected for all vector ${\bf a}\in\mathbb{N}^{2d+1}$. We distinguish the following cases:

{\it Case 1}: $\{x_{1},\ldots, x_{d}\}$, $\{y_{1},\ldots, y_{d}\}\in\Delta_{\b a}(I_{\Delta}^{(d+2)})$. 

\vspace{0.1cm}In this case, since $\Delta_{\bf a}(I_{\Delta}^{(d+2)})=\langle F\in\mathcal{F}(\Delta)\,\,: \sum_{i\notin F}a_{i}\leq d+1\rangle$, we have$$a_{y_{1}}+\cdots a_{y{d}}+a_{z}\leq d+1\mbox{ and }a_{x_{1}}+\cdots a_{x_{d}}+a_{z}\leq d+1.$$Also, we may assume that $a_{x_{1}}\leq\cdots\leq a_{x_{d}}$ and $a_{y_{1}}\leq\cdots\leq a_{y_{d}}.$ We distinguish the following cases:

(1-i): Suppose that $d\geq 4$ and $d$ is even.  

\vspace{0.1cm}In this case, we prove that $\{x_{\frac{d}{2}+1},\ldots, x_{d}, y_{\frac{d}{2}+2},\ldots, y_{d}, z\}$ is a facet of $\Delta_{\bf a}(I_{\Delta}^{(d+2)})$. It is enough to prove that $a_{x_{1}}+\cdots+a_{x_{\frac{d}{2}}}+a_{y_{1}}+\cdots+a_{y_{\frac{d}{2}}+1}\leq d+1$. Now, by the assumption that $a_{x_{1}}+\cdots a_{x_{d}}+a_{z}\leq d+1$ and $a_{x_{1}}\leq\cdots\leq a_{x_{d}}$, by considering the average, we have 
$$\dfrac{a_{x_{1}}+\cdots+a_{x_{\frac{d}{2}}}}{\frac{d}{2}}\leq \dfrac{a_{x_{1}}+\cdots+a_{x_{d}}}{d}\leq \dfrac{d+1}{d}.$$
Therefore, we obtain that $a_{x_{1}}+\cdots+a_{x_{\frac{d}{2}}}\leq\frac{d}{2}\times\frac{d+1}{d}=\frac{d+1}{2}$. Similarly, by considering the average, we have 
$$\dfrac{a_{y_{1}}+\cdots+a_{y_{\frac{d}{2}+1}}}{\frac{d}{2}}\leq \dfrac{a_{y_{1}}+\cdots+a_{y_{d}}}{d}\leq\dfrac{d+1}{d}.$$Therefore, we obtain that $a_{y_{1}}+\cdots+a_{y_{\frac{d}{2}+1}}\leq(\frac{d}{2}+1)\frac{d+1}{d}=\frac{d}{2}+1+\frac{d+2}{2d}$. Now since $d\geq 4$, $\frac{d+2}{2d}<1$. As a consequence, it follows that $a_{x_{\frac{d}{2}+1}}+\cdots+a_{x_{\frac{d}{2}}}+a_{y_{1}}+\cdots+a_{y_{\frac{d}{2}+1}}\leq d+1$, and hence $\{x_{\frac{d}{2}+1},\ldots, x_{d}, y_{\frac{d}{2}+2},\ldots, y_{d}, z\}$ is a facet of $\Delta_{\bf a}(I_{\Delta}^{(d+2)})$.

(1-ii): Suppose that $d\geq 3$ and $d$ is odd.  

\vspace{0.1cm}In this case, we prove that $\{x_{\frac{d+3}{2}},\ldots, x_{d}, y_{\frac{d+1}{2}},\ldots, y_{d}, z\}$ is a facet of $\Delta_{\bf a}(I_{\Delta}^{(d+2)})$. It is enough to prove that $a_{x_{1}}+\cdots+a_{x_{\frac{d+1}{2}}}+a_{y_{1}}+\cdots+a_{y_{\frac{d+1}{2}}}\leq d+1$. By the assumption that $a_{x_{1}}+\cdots a_{x_{d}}+a_{z}\leq d+1$ and $a_{x_{1}}\leq\cdots\leq a_{x_{d}}$, by considering the average, we have 
$$\dfrac{a_{x_{1}}+\cdots+a_{x_{\frac{d+1}{2}}}}{\frac{d+1}{2}}\leq \dfrac{a_{x_{1}}+\cdots+a_{x_{d}}}{d}\leq \dfrac{d+1}{d}.$$Therefore, we obtain that $a_{x_{1}}+\cdots+a_{x_{\frac{d+1}{2}}}\leq\frac{d+1}{2}\times\frac{d+1}{d}=\frac{d+1}{2}+\frac{d+1}{2d}$. Now since $d\geq 3$, $\frac{d+1}{2d}<1$. Hence $a_{x_{1}}+\cdots+a_{x_{\frac{d+1}{2}}}\leq\frac{d+1}{2}$. Similarly, we obtain $a_{y_{1}}+\cdots+a_{y_{\frac{d+1}{2}}}\leq\frac{d+1}{2}$. As a consequence, it follows that $a_{x_{1}}+\cdots+a_{x_{\frac{d+1}{2}}}+a_{y_{1}}+\cdots+a_{y_{\frac{d+1}{2}}}\leq d+1$, and hence $\{x_{\frac{d+3}{2}},\ldots, x_{d}, y_{\frac{d+1}{2}},\ldots, y_{d}, z\}$ is a facet of $\Delta_{\bf a}(I_{\Delta}^{(d+2)})$.

{\it Case 2}: $\{x_{1},\ldots, x_{d}\}\in\Delta_{\b a}(I_{\Delta}^{(d+2)})$ and $\{y_{1},\ldots, y_{d}\}\notin\Delta_{\b a}(I_{\Delta}^{(d+2)})$. 

\vspace{0.1cm}In this case, if there is no facet of the form (5) in $\Delta_{\bf a}(I_{\Delta}^{(d+2)})$, then $\Delta_{\bf a}(I_{\Delta}^{(d+2)})$ is clearly connected. Hence, we suppose that there exists a facet $\{y_{1},\ldots, y_{d}, z\}\setminus\{y_{i}\}$ of $\Delta_{\bf a}(I_{\Delta}^{(d+2)})$. Then since $\Delta_{\bf a}(I_{\Delta}^{(d+2)})=\langle F\in\mathcal{F}(\Delta)\,\,: \sum_{i\notin F}a_{i}\leq d+1\rangle$, we have $$a_{y_{1}}+\cdots+a_{y_{d}}+a_{z}\leq d+1\mbox{ and }a_{x_{1}}+\cdots+a_{x_{d}}+a_{y_{i}}\leq d+1.$$Also, we may assume that $a_{x_{1}}\leq\cdots\leq a_{x_{d}}$ and $a_{y_{1}}\leq\cdots\leq a_{y_{d}}.$ We now prove $\{x_{d-i+1},\ldots, x_{d}, y_{i+1},\ldots, y_{d}, z\}$ is a facet of $\Delta_{a}(I_{\Delta}^{(d+2)})$. By considering the average, we have 
$$\dfrac{a_{x_{1}}+\cdots+a_{x_{d-i}}}{d-i}\leq\dfrac{a_{x_{1}}+\cdots+a_{x_{d}}}{d}\leq\dfrac{d+1}{d}.$$Hence we obtain that $a_{y_{1}}+\cdots+a_{y_{d-i}}\leq\frac{d+1}{d}(d-i)=d-i+1-\frac{i}{d}$. Similarly, we obtain $a_{y_{1}}+\cdots+a_{y_{i}}\leq i+\frac{i}{d}$. As a consequence, it follows that $$a_{x_{1}}+\cdots+a_{x_{d-i}}+a_{y_{1}}+\cdots+a_{y_{i}}\leq\left(i+\frac{i}{d}\right)+\left(d-i+1-\frac{i}{d}\right)=d+1,$$which implies that $\{x_{d-i+1},\ldots, x_{d}, y_{i+1},\ldots, y_{d}, z\}$ is a facet of $\Delta_{a}(I_{\Delta}^{(d+2)})$. 

{\it Case 3}: $\{x_{1},\ldots, x_{d}\}\notin\Delta_{\b a}(I_{\Delta}^{(d+2)})$ and $\{y_{1},\ldots, y_{d}\}\in\Delta_{\b a}(I_{\Delta}^{(d+2)})$. 

\vspace{0.1cm}Similarly to Case 2, we can prove $\{x_{i+1},\ldots, x_{d}, y_{d-i+1},\ldots, y_{d}\}$ is a facet of $\Delta_{\b a}(I_{\Delta}^{(d+2)})$. 

$(4):$ From $(3)$, we have $H_{\mathfrak{m}}^{1}(S/I_{\Delta}^{(d)})\neq0$. Then, since $\dim K_{S/I_{\Delta}^{(d)}}^{1}=0$, we see that $S_{2}\mbox{-}{\rm depth}(S/I_{\Delta}^{(d)})=1$. Moreover, we obtain  that $$S_{2}\mbox{-}{\rm depth}(S/I_{\Delta}^{(d+2)})\geq{\rm depth}(S/I_{\Delta}^{(d+2)})\geq2,$$as required. 
\end{proof}

We prove that the Serre depth for $(S_{2})$ and the depth on the symbolic powers satisfy a non-increasing property when the dimension of a simplicial complex is one.
 
\begin{thm}\label{non-increasing for dim1}
Let $\Delta$ be a simplicial complex with $\dim \Delta=1$ and $\ell\geq1$. Then we have $${\rm depth}(S/I_{\Delta}^{(\ell)})\geq{\rm depth}(S/I_{\Delta}^{(\ell+1)}).$$Moreover, if $\Delta$ is pure, then we have$$S_{2}\mbox{-}{\rm depth}(S/I_{\Delta}^{(\ell)})\geq S_{2}\mbox{-}{\rm depth}(S/I_{\Delta}^{(\ell+1)}).$$
\end{thm}
\begin{proof}
By the assumption that $\dim \Delta=1$, it is enough to prove that ${\rm depth}(S/I_{\Delta}^{(\ell)})=1$ implies ${\rm depth}(S/I_{\Delta}^{(\ell+1)})=1$. We suppose that ${\rm depth}(S/I_{\Delta}^{(\ell)})=1$. Since $H_{\mathfrak{m}}^{1}(S/I_{\Delta}^{(\ell)})\neq0$, there exists a vector ${\bf a}\in\mathbb{N}^{n}$ such that $\Delta_{\bf a}$ is disconnected. Then there exist facets $F_{0}, G_{0}$ of $\Delta_{\bf a}$ with $F_{0}\cap G_{0}=\emptyset$ such that there are no facets $F, G$ of $\Delta$ satisfying $F_{0}\cap F\neq \emptyset$, $G_{0}\cap G\neq\emptyset$ and $F\cap G\neq\emptyset$. Then we can write $F_{0}=\{x_{1}, x_{2}\}$ and $G_{0}=\{y_{1}, y_{2}\}$ since $\dim \Delta=1$. Fix a vector ${\bf a}=(a_{x_{1}}, a_{x_{2}}, a_{y_{1}}, a_{y_{2}}, a_{z_{1}},\ldots, a_{z_{n-4}})$. We distinguish the following cases:

\vspace{0.1cm}
{\it Case 1}: There exist $i, j$ such that $\{x_{i}, y_{j}\}\notin\Delta$. 

\vspace{0.1cm}We may assume that $i=1$ and $j=1$. Set $${\bf a}^{\prime}=(a_{x_{1}}+1, a_{x_{2}}, a_{y_{1}}+1, a_{y_{2}}, a_{z_{1}},\ldots, a_{z_{n-4}}).$$Then, since $\Delta_{{\bf a}^{\prime}}(I_{\Delta}^{(\ell+1)})=\langle F\in\mathcal{F}(\Delta)\,\,: \sum_{i\notin F}a_{i}^{\prime}\leq\ell\rangle$, $F_{0}$ and $G_{0}$ is a facet of $\Delta_{{\bf a}^{\prime}}(I_{\Delta}^{(\ell+1)})$ and $\Delta_{{\bf a}^{\prime}}(I_{\Delta}^{(\ell)})=\langle F\in\mathcal{F}(\Delta_{a}(I_{\Delta}^{(\ell+1)}))\,\,: x_{1}\in F\mbox{ or }y_{1}\in F\rangle$, that is, $\Delta_{{\bf a}^{\prime}}(I_{\Delta}^{(\ell+1)})$ is a subcomplex of $\Delta_{\bf a}(I_{\Delta}^{(\ell)})$ which contains $F_{0}$ and $G_{0}$ as facets. Therefore, by the assumption that $\Delta_{\bf a}(I_{\Delta}^{(\ell)})$ is disconnected by $F_{0}$ and $G_{0}$, $\Delta_{{\bf a}^{\prime}}(I_{\Delta}^{(\ell+1)})$ is disconnected, and hence we have ${\rm depth}(S/I_{\Delta}^{(\ell+1)})=1$. 

\vspace{0.1cm}
{\it Case 2}: For any $i$ and $j$, $\{x_{i}, y_{j}\}\in\Delta$.  

\vspace{0.1cm}In this case, since $F_{0}, G_{0}\in\Delta_{\bf a}(I_{\Delta}^{(\ell)})=\langle F\in\Delta\,\,: \sum_{i\notin F}a_{i}\leq\ell-1\rangle$, we have $$a_{y_{1}}+a_{y_{2}}+\sum_{1\leq i\leq n-4}a_{z_{i}}\leq\ell-1\mbox{ and }a_{x_{1}}+a_{x_{2}}+\sum_{1\leq i\leq n-4}a_{z_{i}}\leq\ell-1.$$We may assume that $a_{x_{1}}\leq a_{x_{2}}$ and $a_{y_{1}}\leq a_{y_{2}}$. By the above inequalities, we have $a_{x_{1}}+a_{y_{1}}+\sum_{1\leq i\leq n-4}a_{z_{i}}\leq\ell-1$. Then since $\{x_{2}, y_{2}\}$ is a facet of $\Delta$, one can see that $\{x_{2}, y_{2}\}$ is also a facet of $\Delta_{\bf a}(I_{\Delta}^{(\ell)})$, which contradicts to the fact that $F_{0}$ and $G_{0}$ are disconnected in $\Delta_{\bf a}(I_{\Delta}^{(\ell)})$. 

Finally, we prove the desired inequality for the Serre depth. If $S_{2}\mbox{-}{\rm depth}(S/I_{\Delta}^{(\ell)})=2$, then since $\dim\Delta=1$, the desired inequality clearly holds. Hence we may assume that $S_{2}\mbox{-}{\rm depth}(S/I_{\Delta}^{(\ell)})=1$. Then $S_{2}\mbox{-}{\rm depth}(S/I_{\Delta}^{(\ell)})\geq{\rm depth}(S/I_{\Delta}^{(\ell)})$, in general, we obtain that $S_{2}\mbox{-}{\rm depth}(S/I_{\Delta}^{(\ell)})={\rm depth}(S/I_{\Delta}^{(\ell)})=1$ for all $\ell\geq1$, as required. 
\end{proof}

As a corollary, we give a classification of the Serre depth for $(S_{2})$ and the depth on the symbolic powers of Stanley--Reisner ideals of one-dimensional simplicial complexes.

\begin{cor}\label{classify of the depth}
Let $\Delta$ be a simplicial complex with $\dim \Delta=1$. Then the sequence of the depth on symbolic powers $({\rm depth}(S/I_{\Delta}^{(\ell)}))_{\ell\geq1}$ is as follows: 
$$({\rm depth}(S/I_{\Delta}^{(\ell)}))_{\ell\geq1}=
\begin{cases}
(2,2,2,\ldots) & \mbox{ if }\Delta\mbox{ is matroid}, \\
(2,2,1,\ldots) & \mbox{ if }{\rm diam}\Delta\leq2\mbox{ and }\Delta\mbox{ is not matroid}, \\
(2,1,1,\ldots) & \mbox{ if }3\leq{\rm diam}\Delta<\infty, \\
(1,1,1,\ldots) & \mbox{ if }{\rm diam}\Delta=\infty. 
\end{cases}
$$Moreover, if $\Delta$ is pure, then we have $(S_{2}\mbox{-}{\rm depth}(S/I_{\Delta}^{(\ell)}))_{\ell\geq1}=({\rm depth}(S/I_{\Delta}^{(\ell)}))_{\ell\geq1}$. 
\end{cor}
\begin{proof}
From \cite[Theorem 3.5]{mt1} or \cite[Theorem 2.1]{v0}, if $\Delta$ is matroid, $S/I_{\Delta}^{(\ell)}$ is Cohen--Macaulay for all $\ell\geq1$ and hence $({\rm depth}(S/I_{\Delta}^{(\ell)}))_{\ell\geq1}=(2,2,2,\ldots)$.  Also, from \cite[Theorem 3.2]{rty}, it is known that ${\rm depth}(S/I_{\Delta}^{(2)})\geq2$ if and only if ${\rm diam}\Delta\leq2$. Hence by Proposition \ref{non-increasing for dim1}, we have $({\rm depth}(S/I_{\Delta}^{(\ell)}))_{\ell\geq1}=(2,2,1,\ldots)$. Moreover, if ${\rm diam}\Delta<\infty$, that is, $\Delta$ is connected, then $\Delta$ is Cohen--Macaulay and hence $({\rm depth}(S/I_{\Delta}^{(\ell)}))_{\ell\geq1}=(2,1,1,\ldots)$, which completes the proof. Moreover, if $\Delta$ is pure, then, by the latter part of the proof in Theorem \ref{non-increasing for dim1}, the assertion follows. 
\end{proof}

Moreover, we prove that a non-increasing property holds for the Serre depth for $(S_{2})$ and the depth, if the number of vertices of a simplicial complex is less than or equal to 5. 

\begin{cor}\label{cor of dim1 case}
Let $\Delta$ be a simplicial complex on the vertex set $[n]$. If $n\leq 5$, then we have$${\rm depth}(S/I_{\Delta}^{(\ell)})\geq{\rm depth}(S/I_{\Delta}^{(\ell+1)})\mbox{ for all }\ell\geq1.$$Moreover, if $\Delta$ is pure, then we have$$S_{2}\mbox{-}{\rm depth}(S/I_{\Delta}^{(\ell)})\geq S_{2}\mbox{-}{\rm depth}(S/I_{\Delta}^{(\ell+1)}).$$
\end{cor}
\begin{proof}
In this case, we have $\dim S/I_{\Delta}^{(\ell)}\leq2$ or ${\rm ht}(I_{\Delta}^{(\ell)})\leq2$, and hence, by Theorem \ref{non-increasing for dim1}, Proposition \ref{non-increasing of cover ideals} and \cite[Theorem 3.2]{hktt}, the assertion follows. 
\end{proof}

If $n\geq6$ in the situation of Corollary \ref{cor of dim1 case}, a non-increasing property of the Serre depth for $(S_{2})$ and the depth does not necessarily hold. The following example illustrates this.

\begin{ex}\label{CE2}
Let $\Delta$ be a simplicial complex on the vertex set $\{x_{1},\ldots, x_{6}\}$ whose facets are $\{x_{1}, x_{2}, x_{3}\},$ $\{x_{1}, x_{2}, x_{4}\},$ $\{x_{1}, x_{3}, x_{4}\},$ $\{x_{2}, x_{3}, x_{4}\},$ $\{x_{1}, x_{4}, x_{5}\},$ $\{x_{2}, x_{4}, x_{5}\},$ $\{x_{3}, x_{4}, x_{5}\},$ $\{x_{4}, x_{5}, x_{6}\}.$ By using computer calculating, we see that 
\begin{enumerate}
\item ${\rm depth}(S/I_{\Delta}^{(\ell)})=2$ and $S_{2}\mbox{-}{\rm depth}(S/I_{\Delta}^{(\ell)})\geq2$ for $1\leq \ell\leq 6$
\item ${\rm depth}(S/I_{\Delta}^{(7)})=1$ and $S_{2}\mbox{-}{\rm depth}(S/I_{\Delta}^{(7)})=1$
\item ${\rm depth}(S/I_{\Delta}^{(8)})=2$ and $S_{2}\mbox{-}{\rm depth}(S/I_{\Delta}^{(8)})\geq2$
\end{enumerate}
\end{ex}

We prove that once the depth becomes 1 for some symbolic power, it remains 1 periodically for higher symbolic powers. Moreover, the following proposition cannot, in general, be obtained by applying \cite[Theorem 2.7]{nt}.
\begin{prop}
Let $\Delta$ be a pure simplicial complex with $\dim \Delta=d-1$ and $\ell\geq1$. If ${\rm depth}(S/I_{\Delta}^{(\ell)})=1$, then we have ${\rm depth}(S/I_{\Delta}^{(\ell+d)})=1$. 
\end{prop}
\begin{proof}
Suppose that ${\rm depth}(S/I_{\Delta}^{(\ell)})=1$. Then, since $H_{\mathfrak{m}}^{1}(S/I_{\Delta}^{(\ell)})\neq0$, there exists a vector ${\bf a}\in\mathbb{N}^{n}$ such that $\Delta_{\bf a}(I_{\Delta}^{(\ell)})$ is disconnected by the facets $F_{0}$ and $G_{0}$. We write $F_{0}=\{x_{1},\ldots, x_{d}\}, G_{0}=\{y_{1},\ldots, y_{d}\}$ and $V(\Delta)\setminus\{F_{0}\cup G_{0}\}=\{z_{1},\ldots, z_{n-2d}\}$, where $V(\Delta)$ is the vertex set of $\Delta$. Also, by relabeling, we may write $${\bf a}=(a_{x_{1}},\ldots, a_{x_{d}}, a_{y_{1}},\ldots, a_{y_{d}}, a_{z_{1}},\ldots, a_{y_{n-2d}}),$$where $a_{x_{i}}$ (resp., $a_{y_{j}}$, $a_{z_{k}}$) is a non-negative integer corresponding to a vertex $x_{i}$ (resp., $y_{j}$,  $z_{k}$). Moreover, we set $${\bf a}^{\prime}=(a_{x_{1}}+1,\ldots, a_{x_{d}}+1, a_{y_{1}}+1,\ldots, a_{y_{d}}+1, a_{z_{1}},\ldots, a_{y_{n-2d}}).$$We now prove that $\Delta_{{\bf a}^{\prime}}(I_{\Delta}^{(\ell+d)})$ is disconnected. Since $F_{0}$ is a facet of $\Delta_{\bf a}(I_{\Delta}^{(\ell)})$, we have $\sum_{1\leq j\leq d}a_{y_{j}}+\sum_{1\leq k\leq n-2d}a_{z_{k}}\leq \ell-1$ and hence we obtain that$$\displaystyle\sum_{i\notin F_{0}}a_{i}^{\prime}=\displaystyle\sum_{1\leq j\leq d}(a_{y_{j}}+1)+\displaystyle\sum_{1\leq k\leq n-2d}a_{z_{k}}\leq \ell+d-1.$$ It follows that $F_{0}$ is a facet of $\Delta_{{\bf a}^{\prime}}(I_{\Delta}^{(\ell+d)})$. The same argument holds for $G_{0}$ as well. Fix a facet $H$ of $\Delta_{{\bf a}^{\prime}}(I_{\Delta}^{(\ell+d)})$. Then, by Lemma \ref{symbolic power cpx}, we have $\sum_{i\notin H}a_{i}^{\prime}\leq\ell+d-1$ and hence $\sum_{i\notin H}a_{i}\leq\left(\sum_{i\notin H}a_{i}^{\prime}\right)-d\leq(\ell+d-1)=\ell-1.$ It follows that $H$ is a facet of $\Delta_{\bf a}(I_{\Delta}^{(\ell)})$. Therefore, since $\Delta_{{\bf a}^{\prime}}(I_{\Delta}^{(\ell+d)})$ is a subcomplex of 
$\Delta_{{\bf a}^{\prime}}(I_{\Delta}^{(\ell)})$ and $\Delta_{\bf a}(I_{\Delta}^{(\ell)})$ is disconnected by $F_{0}$ and $G_{0}$, we see that $\Delta_{{\bf a}^{\prime}}(I_{\Delta}^{(\ell+d)})$ is disconnected, as required. 
\end{proof}

Also, we prove that for certain special simplicial complex, once the depth of a symbolic power reaches 1, it remains 1 for all higher powers.

\begin{prop}\label{special simplicial complex}
Let $\Delta$ be a pure simplicial complex with $\dim \Delta = d-1$. Suppose that $f_{1}<2\binom{d}{2}+d^2$, where $f_{1}$ is the number of 1-dimensional faces of $\Delta$. If ${\rm depth}(S/I_{\Delta}^{(\ell)})=1$ for some $\ell\geq1$, then we have ${\rm depth}(S/I_{\Delta}^{(\ell+1)})=1$. 
\end{prop}
\begin{proof}
Let $\{x_{1},\ldots, x_{n}\}$ be the vertices of $\Delta$. Suppose that ${\rm depth}(S/I_{\Delta}^{(\ell)})=1$ for some $\ell\geq1$. Since $H_{\mathfrak{m}}^{1}(S/I_{\Delta}^{(\ell)})\neq0$, by Theorem \ref{Takayama}, there exists a vector ${\bf a}=(a_{1},\ldots, a_{n})\in\mathbb{N}^{n}$ such that $\Delta_{\bf a}(I_{\Delta}^{(\ell)})$ is disconnected. Hence, there exist facets $F$ and $G$ of $\Delta_{\bf a}(I_{\Delta}^{(\ell)})$ are disconnected. Since $f_{1}<2\binom{d}{2}+d^2$, there exist $x_{i}\in F$ and $x_{j}\in G$ such that $\{x_{i}, x_{j}\}\notin\Delta$. We now prove that $\Delta_{{\bf a}^{\prime}}(I_{\Delta}^{(\ell+1)})$ is disconnected by $F$ and $G$, where ${\bf a}^{\prime}=(a_{1}^{\prime},\ldots, a_{n}^{\prime})$ is defined by $a_{i}^{\prime}=a_{i}+1, a_{j}^{\prime}=a_{j}+1$ and $a_{k}^{\prime}=a_{k}$ for any $k\neq i, j$. By Lemma \ref{symbolic power cpx}, $\Delta_{{\bf a}^{\prime}}(I_{\Delta}^{(\ell+1)})$ is the subcomplex of $\Delta_{\bf a}(I_{\Delta}^{(\ell)})$ whose facets do not contain both $x_{i}$ and $x_{j}$. In particular, $F$ and $G$ are facets of $\Delta_{{\bf a}^{\prime}}(I_{\Delta}^{(\ell+1)})$, and by the assumption that $F$ and $G$ are disconnected in $\Delta_{\bf a}(I_{\Delta}^{(\ell+1)})$, the desired result follows.  
\end{proof}

Moreover, based on \cite[Theorem 3.6]{sf3}, we discuss the stability of the Serre depth for $(S_{2})$ on the symbolic powers of Stanley–Reisner ideals. To this end, we prepare the following lemmas: 

\begin{lemma}\label{lemma1}
Let $\Delta$ be a pure simplicial complex. Suppose that $m$ and $k$ are positive integers. Then for every integer $j$ with $m-k\leq j\leq m$, we have 
$$S_{2}\mbox{-}{\rm depth}(S/I_{\Delta}^{(m)})\geq S_{2}\mbox{-}{\rm depth}(S/I_{\Delta}^{(km+j)})$$
\end{lemma}
\begin{proof}
By Proposition \ref{S_2-depth alg}, we may assume that $S_{2}\mbox{-}{\rm depth}(S/I_{\Delta}^{(m)})=|F|+1$ for some $F\in\Delta$ such that $H_{\mathfrak{m}}^{1}(\Bbbk[V({\rm link}_{\Delta}F)]/I_{{\rm link}_{\Delta}F}^{(m)})\neq0$, that is, we have $${\rm depth}(\Bbbk[V({\rm link}_{\Delta}F)]/I_{{\rm link}_{\Delta}F}^{(m)})=1.$$ From \cite[Theorem 3.4]{mn}, we see that ${\rm depth}(\Bbbk[V({\rm link}_{\Delta}F)]/I_{{\rm link}_{\Delta}F}^{(km+j)})=1$. Hence, we obtain that $H_{\mathfrak{m}}^{1}(\Bbbk[V({\rm link}_{\Delta}F)]/I_{{\rm link}_{\Delta}F}^{(km+j)})\neq0$. Therefore, again by \cite[Proposition 4.10]{mt}, we have $S_{2}\mbox{-}{\rm depth}(S/I_{\Delta}^{(m)})=|F|+1\geq S_{2}\mbox{-}{\rm depth}(S/I_{\Delta}^{(km+j)})$.
\end{proof}

\begin{lemma}\label{lemma2}
For a  pure simplicial complex $\Delta$ and $m,k\geq1$, we have 
$$S_{2}\mbox{-}{\rm depth}(S/I_{\Delta}^{(m)})\geq S_{2}\mbox{-}{\rm depth}(S/I_{\Delta}^{(mk)}).$$
\end{lemma}
\begin{proof}
By Lemma \ref{lemma1}, we see that 
$$S_{2}\mbox{-}{\rm depth}(S/I_{\Delta}^{(m)})\geq S_{2}\mbox{-}{\rm depth}(S/I_{\Delta}^{((k-1)m+m)})=S_{2}\mbox{-}{\rm depth}(S/I_{\Delta}^{(mk)}).$$
\end{proof}

\begin{thm}\label{convergent}
For a pure simplicial complex $\Delta$, the sequence $\{S_{2}\mbox{-}{\rm depth}(S/I_{\Delta}^{(\ell)})\}_{\ell\geq1}^{\infty}$ is convergent. Moreover, we have $$\min_{\ell}\{S_{2}\mbox{-}{\rm depth}(S/I_{\Delta}^{(\ell)})\}=\lim_{\ell\rightarrow\infty}S_{2}\mbox{-}{\rm depth}(S/I_{\Delta}^{(\ell)}).$$
\end{thm}
\begin{proof}
By considering, the Serre depth for $(S_{2})$ as the depth in the proof of \cite[Theorem 3.6]{sf3}, we obtain that the desired equality. 
\end{proof}


\section{The Serre depth on edge and cover ideals}
In this section, we treat edge ideals and cover ideals. First, we reformulate a result given in \cite[Corollary 2.2]{th}. To this end, let us recall the following result by Pournaki et al.

\begin{cor}\cite[Corollary 3.15]{ppty2}\label{very}
Let $G$ be a graph. If ${\rm pd}(S/J(G))={\rm im}(G)+1$ and set $d_{j}=\dim K_{S/J(G)}^{j}$ for all $j$. Then, 
$$d_{j}=
\begin{cases}
-\infty & \mbox{ if }\hspace{0.5cm}0\leq j<{\rm depth}(S/J(G)), \\
2j-(n-2) & \mbox{ if }\hspace{0.5cm}{\rm depth}(S/J(G))\leq j\leq n-2
\end{cases}
$$
\end{cor}

By using this, we reformulate a result given in \cite[Corollary 2.2]{th} and give the Serre depth for $(S_{r})$ on cover ideals of graphs belong to special class. 

\begin{prop}\label{betti table}
Let $G$ be a graph on the vertex set $[n]$. If ${\rm reg}(S/I(G))={\rm im}(G)$, then we have 
$$\max\{j\,\,: \beta_{i,i+j}(S/I(G))\neq 0\}=i\mbox{ for all }i\leq{\rm im}(G)$$
In particular, $S_{r}\mbox{-}{\rm depth}(S/J(G))=n-r-1$. 
\end{prop}
\begin{proof}
From Corollary \ref{very}, we have $S_{r}\mbox{-}{\rm depth}(S/J(G))=n-r-1$ if $r\leq{\rm im}(G)$. Indeed, for $r\leq{\rm im}(G)$, since $n-r-1\geq n-{\rm pd}(S/J(G))={\rm depth}(S/J(G))$, by Corollary \ref{very}, we have 
$$\dim K_{S/J(G)}^{n-r-1}=2(n-r-1)-(n-2)=n-2r=(n-r-1)-r+1,$$
$S_{r}\mbox{-}{\rm depth}(S/J(G))\leq n-r-1$. Also, we suppose that $S_{r}\mbox{-}{\rm depth}(S/J(G))<n-r-1$, then we set $k=S_{r}\mbox{-}{\rm depth}(S/J(G))$. 
By Corollary \ref{very}, we may assume that ${\rm depth}(S/J(G))\leq k$. 
Then, $2k-(n-2)=\dim K_{S/J(G)}^{k}\geq k-r+1$ and hence, $n-r-1>k\geq n-r-1$, which is a contradiction. Hence, from Corollary \ref{gene}, we have 
$${\rm reg}_{\leq r-1}(I(G))-2=(n-2)-(n-r-1).$$
Thus, we obtain that 
$$\max\{j\,\,:\beta_{i,i+j}(I(G))\neq0\mbox{ for some }i\leq r-1\}={\rm reg}_{\leq r-1}(I(G))=r+1$$
Hence, there exists $i\leq r$ such that $\beta_{i,i+r+1}(I(G))\neq 0$. If $i\neq r$, then by continuing above discussion, we have 
$$\max\{j\,\,:\beta_{\ell,\ell+j}(I(G))\neq0\mbox{ for some }\ell\leq i-1\}={\rm reg}_{\leq i-1}(I(G))=i+1<r+1.$$
This contradicts to $\beta_{i,i+r+1}(I(G))\neq 0$, which leads to the conclusion. 
Moreover, from Corollary \ref{gene}, $$r+1={\rm reg}_{\leq r-1}(I(G))=n-S_{r}\mbox{-}{\rm depth}(S/J(G)),$$which completes the proof. 
\end{proof}

In particular, for the Serre depth for $(S_{2})$ on cover ideals, we have the following statement:

\begin{prop}
For a graph $G$, if there exist edges $e, e^{\prime}$ that are 3-disjoint, then $S_{2}\mbox{-}{\rm depth}(S/J(G))=n-3$ and otherwise $S_{2}\mbox{-}{\rm depth}(S/J(G))=n-2$. In particular, the cover ideal $J(G)$ satisfies Serre's condition $(S_{2})$ if and only if there are no edges $e, e^{\prime}$ that are 3-disjoint. 
\end{prop}
\begin{proof}
From Theorem \ref{eq dual} and the fact that there exist edges $e, e^{\prime}$ that are 3-disjoint if and only if $\beta_{2, 4}(S/I(G))\neq 0$, the assertion follows.  
\end{proof}

We prove that the Serre depth for $(S_{2})$ on the symbolic powers of edge ideals and cover ideals satisfy a non-increasing property. To this end, we provide the following lemma, which is a similar result of \cite[Theorem 3.5]{hlt}: 

\begin{lemma}\label{lem of edge ideals}
Let $G$ be a well-covered graph. If $H_{\mathfrak{m}}^{1}(S/I(G)^{(\ell)})\neq 0$, then we have $H_{\mathfrak{m}}^{1}(S/I(G)^{(\ell+1)})\neq 0$ for any $\ell\geq1$. 
\end{lemma}
\begin{proof}
We set $V(G)=\{x_{1},\ldots, x_{n}\}$. Suppose that $H_{\mathfrak{m}}^{1}(S/I(G)^{(\ell)})\neq 0$. Then there exist facets $F, F^{\prime}$ of $\Delta(G)$ and a vector ${\bf a}\in\mathbb{N}^{|V(G)|}$ such that $F$ and $F^{\prime}$ are disconnected in $\Delta_{\bf a}(I(G)^{(\ell)})$. Suppose that $e\in\Delta(G)$ for any $e\in V(G)\times V(G)$ with $e\cap F\neq\emptyset$ and $e\cap F^{\prime}\neq\emptyset$. Fix $u\in F^{\prime}$. Then, for any $v\in F$, we have $\{u, v\}\cap F\neq\emptyset$ and $\{u, v\}\cap F^{\prime}\neq\emptyset$, and hence, we obtain that $\{u, v\}\in\Delta(G)$. This implies that $F\cup\{u\}\in\Delta(G)$, which contradicts to the face that $F$ is a facet of $\Delta(G)$. Therefore, there exists an edge $e=\{x_{i}, x_{j}\}\in E(G)$ such that $e\cap F\neq\emptyset$ and $e\cap F^{\prime}\neq\emptyset$. Therefore, it follows in the same way as in the latter part of the proof of Proposition \ref{special simplicial complex}. 
\end{proof}

We now prove that a non-increasing property holds for the Serre depth for $(S_{2})$ on the symbolic powers of the edge and cover ideals of arbitrary graphs.

\begin{thm}\label{non-increasing of edge ideals}
Let $G$ be a well-covered graph. Then, we have $$S_{2}\mbox{-}{\rm depth}(S/I(G)^{(\ell)})\geq S_{2}\mbox{-}{\rm depth}(S/I(G)^{(\ell+1)})\mbox{ for all }\ell.$$
\end{thm}
\begin{proof}
Fix $\ell\geq1$ and take a face $F$ of $\Delta(G)$ such that $H_{\mathfrak{m}}^{1}(\Bbbk[{\rm link}_{\Delta(G)}F]/I_{{\rm link}_{\Delta(G)}F}^{(\ell)})\neq0$. Notice that ${\rm link}_{\Delta(G)}F$ is the independence complex of $G-N[F]$. By Lemma \ref{lem of edge ideals}, $H_{\mathfrak{m}}^{1}(\Bbbk[{\rm link}_{\Delta(G)}F]/I_{{\rm link}_{\Delta(G)}F}^{(\ell+1)})\neq0$, and by Proposition \ref{S_2-depth alg}, as required. 
\end{proof}

Moreover, we prove that a non-increasing property of the Serre depth for $(S_{2})$ holds for cover ideals of any simple graph as an analogue of \cite[Theorem 3.2]{hktt}. 

\begin{prop}\label{non-increasing of cover ideals}
Let $G$ be a graph. Then, we have $$S_{2}\mbox{-}{\rm depth}(S/J(G)^{(\ell)})\geq S_{2}\mbox{-}{\rm depth}(S/J(G)^{(\ell+1)})\mbox{ for all }\ell.$$
\end{prop}
\begin{proof}
Fix $\ell\geq 1$. By \cite[Theorem 3.2]{hktt}, it is known that $${\rm depth}(S/J(G)^{(\ell)})\geq{\rm depth}(S/J(G)^{(\ell+1)}),$$which implies that if $H_{\mathfrak{m}}^{1}(S/J(G)^{(\ell)})\neq0$, then $H_{\mathfrak{m}}^{1}(S/J(G)^{(\ell+1)})\neq0$. Also, for a face $F\in\Delta(J(G))$, $I_{{\rm link}_{\Delta(J(G))}F}$ is also height two unmixed squarefree monomial ideal, where $\Delta(J(G))$ is the Stanley--Reisner complex of $J(G)$. Hence there exists a graph $G^{\prime}$ such that $I_{{\rm link}_{\Delta(J(G))}F}=J(G^{\prime})$ for any $F\in\Delta(J(G))$ and this  ensure that if $H_{\mathfrak{m}}^{1}(\Bbbk[{\rm link}_{\Delta(J(G))}F]/I_{{\rm link}_{\Delta(J(G))}F}^{(\ell)})\neq0$, then $H_{\mathfrak{m}}^{1}(\Bbbk[{\rm link}_{\Delta(J(G))}F]/I_{{\rm link}_{\Delta(J(G))}F}^{(\ell+1)})\neq0$. Therefore, by Proposition \ref{S_2-depth alg}, the assertion follows. 
\end{proof}

We investigate whether the Serre depth for $(S_{r})$ on the symbolic powers of cover ideals satisfies a non-increasing property for any $r \ge 2$. For this purpose, we prepare a useful tool to analyze the symbolic powers of the cover ideals introduced by Seyed Fakhari.

\begin{const}[\cite{s1}, Construction in Section 3]\label{const}
Let $G$ be a graph with the vertex set $V(G)=\{x_{1},\ldots, x_{n}\}$ and let $\ell\geq1$ be an integer. Then the graph $G_{\ell}$ is defined on new vertices
$$V(G_{\ell})=\{x_{i, p}\,\,: 1\leq i\leq n\mbox{ and }1\leq p\leq\ell\}$$
with the edge set $$E(G_{\ell})=\{\{x_{i,p}, x_{j,q}\}\,\,: \{x_{i}, x_{j}\}\in E(G)\mbox{ and }p+q\leq\ell+1\}.$$
\end{const}

\begin{lemma}[\cite{s1}, Lemma 3.4]\label{polarization of cover}
Let $G$ be a graph. For every integer $\ell\geq1$, the ideal $(J(G)^{(\ell)})^{\rm pol}$ is the cover ideal of $G_{\ell}$. 
\end{lemma}

We prove a non-increasing property of the $S_{r}\mbox{-}{\rm depth}(S/J(G)^{(\ell)})$ for a graph G belongs to a special class. Let $S_{(i)}=\Bbbk[V(G_{i})]$ be a polynomial ring, where $G_{i}$ is the graph defined in Construction \ref{const}. 

\begin{thm}\label{non-increasing for S_r}
Let $G$ be a graph such that ${\rm reg}(S_{(i)}/I(G_{i}))={\rm im}(G_{i})$ for all $i\geq1$ and let $r\geq 2$. 
Then we have $$S_{r}\mbox{-}{\rm depth}(S/J(G)^{(\ell)})\geq S_{r}\mbox{-}{\rm depth}(S/J(G)^{(\ell+1)})\mbox{ for all }\ell\geq 1.$$
\end{thm}
\begin{proof}
Let $[n]$ be the vertex set of $G$. Notice that $S_{(i)}$ is the polynomial ring of $J(G)^{(i)}$ by Lemma \ref{polarization of cover}, and hence $|X_{J(G)^{(i)}}|=|V(G_{i})|-n$ for all $i$. Also, by Lemma \ref{polarization of cover}, $S_{r}\mbox{-}{\rm depth}(S_{(i)}/(J(G)^{(i)})^{\rm pol})=S_{r}\mbox{-}{\rm depth}(S_{(i)}/(J(G_{i})))$ for all $i$. Fix an integer $\ell$ with $\ell\geq1$. Then, by the assumption that ${\rm reg}(S_{(i)}/I(G_{i}))={\rm im}(G_{i})$ for all $i$ and Proposition \ref{betti table}, Theorem \ref{eq dual}, we have $$|V(G_{\ell})|-S_{r}\mbox{-}{\rm depth}(S_{(\ell)}/J(G_{\ell}))-1\leq|V(G_{\ell+1})|-S_{r}\mbox{-}{\rm depth}(S_{(\ell+1)}/J(G_{\ell+1}))-1.$$Also, by Proposition \ref{polarization}, we have $$S_{r}\mbox{-}{\rm depth}(S_{(i)}/(J(G)^{(i)})^{\rm pol})=S_{r}\mbox{-}{\rm depth}(S/J(G)^{(i)}))+|X_{{J(G)}^{(i)}}|\mbox{ for all }i.$$Therefore, we obtain that 
\begin{align*}
-n+S_{r}\mbox{-}{\rm depth}(S/J(G)^{(\ell)})&=-|V(G_{\ell})|+|X_{J(G)^{(\ell)}}|+S_{r}\mbox{-}{\rm depth}(S/(J(G)^{(\ell)})))\\
&=-|V(G_{\ell})|+S_{r}\mbox{-}{\rm depth}(S_{(\ell)}/(J(G)^{(\ell)})^{\rm pol}) \\
&\geq-|V(G_{\ell+1})|+S_{r}\mbox{-}{\rm depth}(S_{(\ell+1)}/(J(G)^{(\ell+1)})^{\rm pol}) \\
&=-|V(G_{\ell+1})|+|X_{J(G)^{(\ell+1)}}|+S_{r}\mbox{-}{\rm depth}(S/(J(G)^{(\ell+1)})) \\
&=-n+S_{r}\mbox{-}{\rm depth}(S/J(G)^{(\ell+1)}),
\end{align*}
which completes the proof. 
\end{proof}

We provide an example of graphs satisfy the conditions of Theorem \ref{non-increasing for S_r}. To this end, let us recall the definition of very well-covered graphs. A graph $G$ without isolated vertices is called {\it very well-covered}, if every minimal vertex cover has the same cardinality which is equal to half the number of vertices. For a very well-covered graph, it is known that the following statement: 
\begin{prop}[\cite{s1}, Proposition 3.1]\label{very well lemma}
Let $G$ be a graph without isolated vertices and $\ell\geq1$ be an integer. 
If $G$ is very well-covered graph, then $G_{\ell}$ is very well-covered too. 
\end{prop}

\begin{cor}
Let $G$ be a very well-covered graph and let $r\geq 2$. 
Then we have $$S_{r}\mbox{-}{\rm depth}(S/J(G)^{(\ell)})\geq S_{r}\mbox{-}{\rm depth}(S/J(G)^{(\ell+1)})\mbox{ for all }\ell\geq 1.$$
\end{cor}
\begin{proof}
By Proposition \ref{very well lemma}, $G_{i}$ is very well-covered and hence, by \cite[Theorem 1.3]{mmcrty}, we obtain that ${\rm reg}(S_{(i)}/I(G_{i}))={\rm im}(G_{i})$ for all $i\geq1$. Therefore, by Theorem \ref{non-increasing for S_r}, the assertion follows. 
\end{proof}

The following question arises naturally from Proposition \ref{non-increasing of cover ideals}, Theorem \ref{non-increasing for S_r} and \cite[Theorem 3.2]{hktt}.

\begin{q}
For a graph $G$ and $r\geq2$, is it true that
$$S_{r}\mbox{-}{\rm depth}(S/J(G)^{(\ell)})\geq S_{r}\mbox{-}{\rm depth}(S/J(G)^{(\ell+1)})\mbox{ for every }\ell\geq 1?$$
\end{q}

Finally, we determine the Serre depth on the edge ideals of very well-covered graphs. To this end, we recall a theorem from \cite[Theorem 3.5]{kpty} that is useful for analyzing very well-covered graphs that are not Cohen--Macaulay. We follow the notation in \cite{kpty}. According to \cite{crt}, we may assume the following condition:

\vspace{0.1cm}
\begin{itemize}
\item[$(\ast)$] $V(H)=X_{[d_0]}\cup Y_{[d_0]}$, where $X_{[d_0]}=\{x_1,\ldots,x_{d_0}\}$ is a minimal vertex cover of $H$ and $Y_{[d_0]}=\{y_1,\ldots,y_{d_0}\}$ is a maximal independent set of $H$ such that $\{x_1y_1,\ldots,x_{d_0}y_{d_0}\} \subseteq E(H)$. Moreover, $x_iy_j \in E(H)$ implies that $i\le j$.
\end{itemize}
To study the structure of non-Cohen--Macaulay very well-covered graphs, a particularly useful result is provided in~\cite[Theorem~3.5]{kpty}. Let $G$ be a graph with $xy\in E(G)$, then, the graph $G^{\prime}$, obtained by replacing the edge $xy$ in $G$ with a complete bipartite graph $K_{i, i}$ is defined as $$V(G')=(V(G)\setminus \{x,y\} )\cup \{x_1,\ldots,x_i\}\cup\{y_1 ,\ldots,y_i\}$$ and  
\[
\begin{array}{lll}
E(G')	& = & E(G_{V(G)\setminus \{x,y\}}) \cup\{x_jy_k\,\,: 1\le j,k\le i\} \\[0.15cm]
& & \hspace{2.65cm}\cup\ \{x_jz\,\,: 1\le j\le i,\ z \in V(G) \setminus \{x,y \},\ xz\in E(G)\} \\[0.15cm]
& & \hspace{2.65cm}\cup\ \{y_jz\,\,: 1\le j\le i,\ z \in V(G) \setminus \{x,y \},\ yz\in E(G)\}.
\end{array}
\]
Let $H$ be a Cohen--Macaulay very well-covered graph with $2d_0$ vertices and assume the condition $(\ast)$. Let $H^{\prime}$ be a graph with the vertex set $$V(H^{\prime})=\bigcup_{i=1}^{d_0}\Big(\big\{ x_{i1},\ldots,x_{in_i}\big\}\cup\big\{y_{i1},\ldots,y_{in_i}\big\}\Big)$$ which is obtained by replacing the edges $x_1y_1,\ldots,x_{d_0}y_{d_0}$ in $H$ with the complete bipartite graphs $K_{n_1,n_1},\ldots,K_{n_{d_0},n_{d_0}}$, respectively. We write $H(n_1,\ldots,n_{d_0})$ for $H^{\prime}$. The following theorem is the result mentioned above.

\begin{thm}[\cite{kpty}, Theorem 3.5]\label{structure}
Let $G$ be a very well-covered graph on the vertex set $X_{[d]}\cup Y_{[d]}$. Then there exist positive integers $n_1,\ldots,n_{d_0}$ with $\sum_{i\in [d_0] }n_i=d$ and a Cohen--Macaulay very well-covered graph $H$ on the vertex set $X_ {[d_0]}\cup Y_{[d_0]}$ such that $G\cong H(n_1,\ldots,n_{d_0})$.
\end{thm}

To give the Serre depth on edge ideals of an arbitrary very well-covered graph, we explain notations introduced in \cite[Section 4]{kpty}. Let $H$ be a Cohen--Macaulay very well-covered graph on the vertex set $X_{[d_{0}]}\cup Y_{[d_{0}]}$ and assume the condition $(\ast)$. Let $P(H)$ be the set of independent set $F$ of $H$ such that $H-N[F]$ consists of some disjoint edges (including an empty graph and a single edges) without isolated vertices. If $F\in P(H)$, then there exists a unique subset $T\subset[d_{0}]$ with $E(H_{T})=E_{T}$ such that ${\rm link}_{\Delta(H)}F=\Delta(H_{T})$, where $E_{T}=\{\{x_{i}, y_{i}\}\,\,: i\in T\}$ and $H_{T}$ is the induced subgraph on $\cup_{i\in T}\{x_{i}, y_{i}\}$. We denote this subset by $T_{F}$. Moreover, let $G=H(n_{1},\ldots, n_{d_{0}})$ be a very well-covered graph. If $F_{0}\in\Delta(H)$, then we set $N_{F_{0}}=\sum_{j}n_{j}$, where the sum runs over all $j$ such that $x_{j}\in F_{0}$ or $y_{j}\in F_{0}$. 

By using Theorem \ref{structure}, we provide the Serre depth on edge ideals of an arbitrary very well-covered graph. 

\begin{thm}\label{S_r very well}
Let $G=H(n_{1},\ldots, n_{d_{0}})$ be a very well-covered graph and $r\geq2$. Then we have
\[
S_{r}\mbox{-}{\rm depth}(S/I(G))
= d - \max \Bigl\{\, N_T - |T|\,\,:
  \substack{
    \displaystyle T \subset [d_0] \text{ such that }  E_T \text{ is an} \\ 
    \displaystyle \text{ induced matching and } |T| \le r-1
  } \,\Bigr\}.
\]
In particular, we have $S_{2}\mbox{-}{\rm depth}(S/I(G))=d-\max\{n_{1},\ldots, n_{d_{0}}\}+1.$
\end{thm}
\begin{proof}
From Hochster's formula for local cohomology modules and \cite[Theorem 4.4]{kpty}, we obtain that 
\begin{align*}
S_{2}\mbox{-}{\rm depth}(S/I(G))&=\min\{i\,\,: H_{\mathfrak{m}}^{i}(S/I(G))_{-j}\neq0\mbox{ for some }j\geq i-r+1\} \\
&=\min\{N_{F_{0}}-|F_{0}|+d_{0}\,\,: F_{0}\in P(H)\mbox{ such that }|F_{0}|\geq d_{0}-r+1\}. 
\end{align*}
Now, since $N_{F_{0}}+N_{T_{F_{0}}}=d$ and $|F_{0}|+|T_{F_{0}}|=d_{0}$, the latter one is equal to the value of $\min\{d-N_{T_{F_{0}}}-|T_{F_{0}}|\,\,: F_{0}\in P(H)\mbox{ such that }|F_{0}|\geq d_{0}-r+1\}$. Therefore, by \cite[Lemma 4.1]{kpty}, the assertion follows. 
\end{proof}

\begin{cor}\label{S_2 very well}
For a very well-covered graph $G$, $G$ satisfies Serre's condition $(S_{2})$ if and only if $G$ is Cohen--Macaulay. 
\end{cor}

\begin{rem}
The result in Corollary \ref{S_2 very well} is actually a known fact, as it follows from \cite[Theorem 0.3]{crt} and the fact that the independence complex $\Delta(G)$ of $G$ satisfies Serre's condition $(S_2)$ implies $\Delta$ is strongly connected by \cite[Corollary 2.4]{h}.
\end{rem}

\section*{Acknowledgement}
The authors would like to express their gratitude to Tatsuya Kataoka for his helpful and valuable comments. 
 The research of Yuji Muta was partially supported by ohmoto-ikueikai and JST SPRING Japan Grant Number JPMJSP2126. 
 The research of Naoki Terai was partially supported by JSPS Grant-in Aid for Scientific Research (C) 24K06670.



\end{document}